\documentclass[11pt]{article}

\pdfoutput=1

\usepackage[utf8]{inputenc}
\usepackage[T1]{fontenc}
\usepackage{lmodern,ulem}
\usepackage[makeroom]{cancel}
\usepackage{pdfsync}

\hsize \textwidth
\advance \hsize by -\marginparwidth
\evensidemargin \oddsidemargin

\usepackage{microtype}

\newcommand{\calG}{{\mathcal G}}
\newcommand{\farc}{\frac}

\usepackage[pagebackref,colorlinks=true,pdfpagemode=none,urlcolor=blue,
linkcolor=blue,citecolor=blue]{hyperref}

\usepackage{amsmath,amsfonts,amssymb,amsthm}
\usepackage{amsthm, amssymb, amsmath, amsfonts, mathrsfs}

\usepackage{mathrsfs}
\usepackage{MnSymbol}
\usepackage{scalerel} 

\usepackage{accents}

\usepackage{bbm}

\newtheorem{theorem}{Theorem}[section]
\newtheorem{lemma}[theorem]{Lemma}

\newtheorem{proposition}[theorem]{Proposition}

\theoremstyle{remark}
\newtheorem{remark}[theorem]{Remark}

\renewenvironment{proof}[1][Proof]{ {\itshape \noindent {#1.}} }{$\Box$
\medskip}

\numberwithin{equation}{section}
\newcommand{\R}{\mathbb{R}}

\newcommand{\Z}{\mathbb{Z}}

\newcommand{\Pb}{\mathbb{P}}

\newcommand{\E}{\mathbb{E}}

\newcommand{\F}{\mathcal{F}}

\newcommand{\G}{\mathcal{G}}

\newcommand{\C}{\mathcal{C}}

\newcommand{\cR}{\mathcal{R}}

\newcommand{\V}{\mathcal{V}}

\newcommand{\Q}{\mathcal{Q}}
\newcommand{\U}{\mathcal{U}}
\newcommand{\cU}{\mathscr{U}}

\newcommand{\eps}{\varepsilon}
\def\les{\lesssim}

\newcommand{\Var}{\mathrm{Var}}

\newcommand{\la}{\langle}
\newcommand{\ra}{\rangle}

\newcommand{\X}{\mathbf{X}}

\newcommand{\Y}{\mathbf{Y}}

\newcommand{\1}{\mathbbm{1}}

\newcommand{\cH}{\mathcal{H}}

\newcommand{\cP}{\mathcal{P}}
\newcommand{\ccR}{\mathscr{R}}

\begin{document}

\title{Fluctuations of the solutions to the KPZ equation\\ in dimensions three and higher}
\author{Alexander Dunlap, Yu Gu, Lenya Ryzhik, Ofer Zeitouni}
\date{}
\maketitle

\begin{abstract}
We prove, using probabilistic techniques and analysis on the Wiener space,  that 
the large scale fluctuations of the KPZ equation in $d\geq 3$ with a small coupling constant, driven by a white in time and colored in space noise,
are given by the Edwards-Wilkinson model. This gives an alternative proof, that avoids
perturbation expansions, to the results of Magnen and Unterberger~\cite{magnen2017diffusive}.
\end{abstract}
\maketitle

\section{Introduction}

\subsection{The main result}

We consider the random heat equation 
\begin{equation}\label{e.maineq}
\partial_t u=\frac12\Delta u+\beta V(t,x)u, \    \  u(0,x)\equiv1,  \   \  x\in\R^d, d\geq 3,
\end{equation}
with a coupling constant $\beta>0$, and a random process $V(t,x)$  that is white in time and colored in space, constructed from a space-time 
white noise $\dot W(t,x)$ defined on the probability space $(\Omega,\mathcal{F},\Pb)$:
\[
V(t,x)=\int_{\R^d} \varphi(x-y)\dot{W}(t,y)dy.
\] 
Here, the non-negative mollifier $ \varphi\ge 0$ is in $\C_c^\infty(\R^d)$, and the product $V(t,x)u$ in \eqref{e.maineq} 
is interpreted in the It\^o sense. The solution of  \eqref{e.maineq} is
then   a continuous random field.  We denote by $R(x)$  the spatial covariance function of $V(t,x)$: 
\begin{equation}\label{e.defR}
R(x)=\int_{\R^d} \varphi(x+y)\varphi(y)dy,
\end{equation}
so we can formally write 
\[
\E[V(t,x)V(s,y)]=\delta(t-s)R(x-y).
\] 
Note that $R\in \C_c^\infty(\R^d)$, and without additional 
loss of generality, we assume   that $R(x)=0$ for~$|x|\ge 1$.

Consider the macroscopically rescaled solution 
\[
u_\eps(t,x)=u(\frac{t}{\eps^2},\frac{x}{\eps}),
\]
which solves 
\[
\partial_t u_\eps=\frac12\Delta u_\eps+\frac{\beta}{\eps^2} V(\frac{t}{\eps^2},\frac{x}{\eps}) u_\eps.
\]
The goal of this paper is to identify the asymptotic fluctuations of 
\begin{equation}\label{nov602}
h_\eps(t,x)=\frac{1}{\beta\eps^{(d-2)/2}}\log u_\eps(t,x)
\end{equation}
 as a random distribution, that is, the asymptotic distribution, as $\eps\to 0$,  of
\begin{equation}
\int_{\R^d} h_\eps(t,x)g(x)dx
\end{equation}
for any test function $g\in \C_c^\infty(\R^d)$.  Note that by the scaling property of the space-time white noise, we have 
\begin{equation}\label{e.lawV}
\frac{1}{\eps^2}V(\frac{t}{\eps^2},\frac{x}{\eps}) \stackrel{\text{law}}{=} \eps^{\frac{d-2}{2}} \dot{W}_\eps(t,x), \quad      \mbox{ with } \quad \dot{W}_\eps(t,x)=\frac{1}{\eps^d}\int_{\R^d} \varphi(\frac{x-y}{\eps}) \dot{W}(t,y)dy.
\end{equation}
By \eqref{e.lawV} and the It\^o formula, it is clear that $h_\eps -\E[h_\eps]\stackrel{\text{law}}{=}\tilde{h}_\eps-\E[\tilde{h}_\eps]$ with 
\begin{equation}\label{e.kpz}
\partial_t \tilde{h}_\eps=\frac12\Delta \tilde{h}_\eps+\frac{1}{2}\beta\eps^{\frac{d-2}{2}}|\nabla \tilde{h}_\eps|^2 + \dot{W}_\eps(t,x),
 \quad \tilde{h}_\eps(0,x)\equiv0.
\end{equation}
Here is the main result of paper. 
\begin{theorem}\label{t.mainth}
  There exists $\beta_0=\beta_0(d,\varphi)$ so that  
for $\beta<\beta_0$, $t>0$ and test function $g\in \C_c^\infty(\R^d)$, we have
\begin{equation}\label{e.mainre}
\int_{\R^d} \left(h_\eps(t,x)-\E[h_\eps(t,x)]\right)g(x)dx
\Rightarrow \int_{\R^d} \U(t,x) g(x)dx
\end{equation}
in distribution as $\eps\to0$, where $\U$ solves the Edwards-Wilkinson equation 
\begin{equation}\label{e.defU}
\partial_t \U=\frac12\Delta \U+ \nu_{\mathrm{eff}}  \dot{W}(t,x), \   \ \U(0,x)\equiv 0,
\end{equation}
with the effective variance
\begin{equation}\label{e.effvar}
\nu_{\mathrm{eff}}^2=\int_{\R^d} R(x) \E_B\left[ \exp\left(\frac12\beta^2\int_0^\infty R(x+B_s)ds\right)\right] dx,
\end{equation}
where $B$ is a standard Brownian motion and $\E_B$ denotes the expectation with respect to it.
\end{theorem}
Theorem~\ref{t.mainth} is a particular case of the following 
more general result. 
\begin{theorem}\label{t.generalf}
Assume that a function $\mathfrak{f}\in \C^2(\R_+)$ satisfies 
\[
|\mathfrak{f}(y)|+|\mathfrak{f}'(y)|+|\mathfrak{f}''(y)| \leq M(y^p+y^{-p}), \quad y\in\R_+
\]
for some $M,p>0$. Then, there exists 
$\beta_0=\beta_0(d,\varphi,p,M)$ so that,
for $\beta<\beta_0$, $t>0$ and test function
$g\in \C_c^\infty(\R^d)$, we have 
\begin{equation}
\frac{1}{\beta\eps^{(d-2)/2}}\int_{\R^d}(\mathfrak{f}(u_\eps(t,x))-\E[\mathfrak{f}(u_\eps(t,x))]) g(x)dx\Rightarrow \sigma_{\mathfrak{f}}\int_{\R^d} \U(t,x)g(x)dx
\end{equation}
in distribution as $\eps\to0$, where $\U$ is a solution to \eqref{e.defU} 
and $\sigma_{\mathfrak{f}}=\E[\mathfrak{f}'(Z_\infty)Z_\infty]$. Here, $Z_\infty$ 
is the positive random variable defined in \eqref{e.zinfinity} below with $\E[Z_\infty]=1$.
\end{theorem}
Throughout the paper, we assume that
$\beta\in(0,\beta_0)$ as in the statements of Theorems \ref{t.mainth} and 
\ref{t.generalf}.

\subsection{The context}
The study of the KPZ equation has witnessed important progress in recent years. A lot of work was done in $d=1$, including 
making sense of the equation without relying on the Hopf-Cole transform 
\cite{gubinelli2015paracontrolled,gubinelli2017kpz,hairer2013solving,hairer2014theory,kupiainen2016renormalization}, proving the weak/strong 
universality conjecture in the one-dimensional KPZ universality class 
\cite{alberts2014intermediate,amir2011probability,bertini1997stochastic,matetski2016kpz}, etc. We refer to the reviews 
\cite{corwin2012kardar,quastel2015one} for a more complete list of references. 
In $d=2$, some relevant results can be found in 
\cite{bertini1998two,caravenna2015universality,CSZ18,CRS18,chatterjee2018constructing,feng2016rescaled,toninelli20172}. 

Our result in $d\geq 3$ can be viewed as a continuation of the previous works on the stochastic heat equation (SHE) \cite{dunlap2018random,GRZ17} and as a counterpart of the recent work of Magnen and Unterberger \cite{magnen2017diffusive},  where the driving force is mollified in both temporal and spatial variables. While the proof in \cite{magnen2017diffusive} is based on a multiscale expansion
and a calculation of  multi-point correlation functions, we present a probabilistic proof using the tools of Malliavin calculus. 

It is well-known that in $d\geq3$ there is a phase transition as a
function of 
the coupling constant~$\beta$, also known as the inverse temperature 
if we view the solution to \eqref{e.maineq} as the partition function of a directed polymer in a random environment, and the behaviors of the solution to 
the equation and the underlying polymers change drastically for different values of $\beta$. There are different notions of critical temperatures in $d
\geq3$ \cite{comets2}. For our analysis, what is particularly important is that we stay deep in the weak-disorder/high-temperature regime where $\beta$ is 
small, the $L^2(\Omega)$ norm of~$u(t,x)$ is bounded uniformly in $t>0$, and the effective variance in \eqref{e.effvar} is finite. For directed polymers, the diffusive behavior was proved in the early 
work of~\cite{bolthausen1989note,imbrie1988diffusion}, and we refer to \cite{comets2} for a review of further developments and 
\cite{mukherjee2017quench,mukherjee2016weak} for the connection to the
stochastic heat equation.

\subsection{Connection to the stochastic heat equation}

We note that the coefficient in front of the nonlinear term in \eqref{e.kpz} is small when $d\geq3$, and if we naively ignore the nonlinear term, 
the limiting equation for $\tilde{h}_\eps$ would be 
\[
\partial_t \bar h=\frac12\Delta \bar h+\nu \dot{W}(t,x),
\]
with 
\[
\nu^2=\int_{\R^d} R(x)dx<\nu_{\mathrm{eff}}^2,
\] 
since $\dot{W}_\eps(t,x)\to \nu \dot{W}(t,x)$ as $\eps\to0$. Thus, Theorem~\ref{t.mainth} shows that the ``optically small'' nonlinear term affects   
the effective variance asymptotically, even in the limit $\eps\to 0$,
when $\beta$ is small but does not change fundamentally the nature of the Edwards-Wilkinson
limit. Let us explain
informally why this happens. As we have shown in~\cite{dunlap2018random}, when $\beta$ is small,
the 
solution of (\ref{e.maineq}), before any space-time rescaling, behaves at large times approximately
as a space-time stationary random field $\Psi(t,x)$ with spatial correlations that decay, when the potential is white in time, as 
\begin{equation}\label{dec902}
\hbox{Cov}({\Psi}(0,0),{\Psi}(0,{y}))\sim \frac{\bar{c}\beta^2\nu_{\mathrm{eff}}^2 }{|y|^{d-2}},~~|y|\gg 1,
\end{equation}
with a universal constant $\bar c$. Thus, $h(t,x)=\log u(t,x)$
is approximately~$\log\Psi(t,x)$ with a similar decay of correlations. The law of large numbers implies then 
that $h(t,x)$ converges as a random distribution (after a spacetime rescaling and integration against a smooth test function) to a constant. The spatial decay rate of correlations in (\ref{dec902})  indicates that, as a distribution, 
$\bar h(t,x)=h(t,x)-\E[h(t,x)]$ is of size~$\eps^{(d-2)/2}$  when integrated against a test function on the scale $\eps^{-1}$. 
In other words, $\bar h(t,x)$ is close to zero as a random distribution but not point-wise. On a technical note,
the slow spatial decay of correlations of~$\Psi(t,x)$ does not allow us to apply the 
central limit theorem directly to integrals on the macroscopic scale~$\eps^{-1}$, 
thinking of the integral as a sum over~$\eps^{-d}$ 
boxes. 
In a sense, the hard work is to incorporate the fast  temporal 
mixing of~$V(t,x)$ into the picture, with the help of the Feynman-Kac formula.
Ignoring this serious technical issue, the correlation structure of~$\log\Psi(t,x)$ then dictates the rescaling by the factor $\eps^{(d-2)/2}$ in~(\ref{nov602})
that, in turn, shows up as the factor $\eps^{(d-2)/2}$ in front of the nonlinear term in \eqref{e.kpz}. However, as~$\bar h(t,x)$ is not close
to a constant pointwise, this term is not ``$\eps$-wise'' small in 
the pointwise  sense, only~``$\beta$-wise'' small, 
and thus makes a non-trivial contribution to the effective
variance,  of order~$o(1)$ as $\beta\to 0$, that survives in the limit~$\eps\to 0$. The 
precise Edwards-Wilkinson nature of the limit comes  
as a combination of the Gaussianity coming from the central  limit
theorem, modulo the technical difficulties discussed above,   and the heat semi-group.

In \cite[Theorem 1.2]{GRZ17}, it was
proved that the fluctuations of $u_\eps$ are given by the same Edwards-Wilkinson model:
\begin{equation}\label{e.she}
\frac{1}{\beta \eps^{(d-2)/2}}\int_{\R^d}(u_\eps(t,x)-\E[u_\eps(t,x)]) g(x)\Rightarrow \int_{\R^d} \U(t,x)g(x)dx, \mbox{ as } \eps\to0,
\end{equation}
which can also be viewed as a special case of Theorem~\ref{t.generalf} with $\mathfrak{f}(y)=y$.
In other words, when viewed as random fields, 
$u_\eps(t,\cdot)$ and $\log u_\eps(t,\cdot)$ have
the same limiting distribution!  While it is unclear at first glance
why this should be the case, the proof in this paper helps illustrate the connection, see the discussion in Remark~\ref{r.kpz-she}.  

Let us make a couple of remarks on Theorem~\ref{t.generalf}. First,
if we take $\mathfrak{f}(y)=\log y-y$ in Theorem~\ref{t.generalf}, 
then
\[
\sigma_{\mathfrak{f}}=1-\E[Z_\infty]=0,
\]
hence
\begin{equation}\label{e.errorgotozero}
\frac{1}{\eps^{(d-2)/2}}\int_{\R^d}( \mathfrak{f}(u_\eps(t,x))-\E[\mathfrak{f}(u_\eps(t,x))]) g(x)dx\to0
\end{equation}
in probability as $\eps\to0$. In other words, we have
\begin{equation}\label{dec602}
\eps^{-(d-2)/2}(\log u_\eps(t,\cdot)-\E[ \log u_\eps(t,\cdot)])\approx \eps^{-(d-2)/2}[u_\eps(t,\cdot)-1]
\end{equation}
as random distributions. We stress that (\ref{dec602}) is not a simple
consequence of  a linearization of $\log y$ around $y=1$.  
For example, for $\mathfrak{f}(y)=y^2$, we have 
\[
\sigma_{\mathfrak{f}}=2\E[Z_\infty^2]>2\E[Z_\infty]^2=2,
\]
which is different from the coefficient obtained by the 
linearization $\mathfrak{f}(y)\approx 2y-1$ near $y=1$. 
Such a linearization fails 
in general precisely because $u_\eps(t,x)$ is not close to $1$ but rather approaches
a stationary solution in the long time limit, as discussed above. 

Second, if we recall the connection between the effective variance
in the Edwards-Wilkinson equation and the decay of correlations rate in (\ref{dec902}) for $\Psi(t,x)$, and note that a similar connection
holds for $\mathfrak{f}(\Psi(t,x))$, 
Theorem~\ref{t.generalf} says that if $\Psi(t,x)$ satisfies (\ref{dec902}), then
the correlations of~$\mathfrak{f}(\Psi(t,x))$    decay as
\begin{equation}\label{dec904}
\hbox{Cov}(\mathfrak{f}({\Psi}(0,0)),\mathfrak{f}({\Psi}(0,{y})))\sim \frac{\bar{c}\beta^2\sigma_{\mathfrak{f}}^2 \nu_{\mathrm{eff}}^2 }{|y|^{d-2}},~~|y|\gg 1,
\end{equation}
with $\sigma_{\mathfrak{f}}=\E[\mathfrak{f}'(\Psi(0,0))\Psi(0,0)]$. 
As explained in Remark~\ref{r.sigmaf} below, this is consistent with $\log\Psi(t,x)$ being a ``Gaussian random field on large scales''.

For a potential that is smooth in both $x$ and $t$ variable, such as 
\[
V(t,x)=\int_{\R^{d+1}}\phi(t-s)\varphi(x-y)\dot{W}(s,y)dyds
\]
 for some $\phi\in\C_c^\infty(\R)$, a result similar to \eqref{e.she} was proved in \cite[Theorem 1.1]{GRZ17}, where the limiting Edwards-Wilkinson 
 equation has also an effective diffusivity:
 \[
 \partial_t \cU=\frac12\nabla\cdot a_{\mathrm{eff}}\nabla \cU+\nu_{\mathrm{eff}} \dot{W}.
 \]
This is consistent with the result in \cite{magnen2017diffusive}, where the same limit is proved for the KPZ equation. The effective diffusion matrix $a_{\mathrm{eff}}$ comes from the temporal correlation of the randomness, 
which was previously discussed in \cite{betz2005central,gubinelli2006gibbs}. In \cite{GRZ17,mukherjee2017central}, as a crucial ingredient of the proof, a Markov chain was constructed to model the evolution of the path increments and the fast mixing of the Markov chain drives a central limit theorem which gives rise to the effective diffusivity. We believe 
that 
the approach developed in this paper, combined with the Markov chain techniques used in \cite{GRZ17}, will give another proof of the result 
in \cite{magnen2017diffusive} for random potentials that are not white in time. We choose to work in the white in time setting since it is technically simpler to
explain, but is also illustrative enough to reveal the main idea of the proof for the 
general case. 

All the aforementioned results are on the asymptotics of the random fields after a spatial averaging. For the pointwise fluctuations, we refer to the recent work \cite{CometsMukh,comets1}. In the weak-coupling regime, some discussion on the pointwise fluctuations can be found in \cite{meerson}.

From the expression of $\nu_{\mathrm{eff}}^2(\beta)$ in \eqref{e.effvar}, we can define
\[
\beta_{\mathrm{critical}}=\sup \{ \beta>0: \nu_{\mathrm{eff}}^2(\beta)<\infty\},
\]
and refer $\beta\in (0,\beta_{\mathrm{critical}})$ as the $L^2-$region. It is unclear if our approach in this paper can be refined to cover the whole range of $\beta\in(0,\beta_{\mathrm{critical}})$. The recent work \cite{CRS18} used a different method and proved a similar result for  $d=2$ in the whole $L^2-$region.

\medskip

\subsection{Notation}
\label{sec-notation}
We use throughout the following notations and conventions.

(i) We use $a \les b$ for $a\leq Cb$ for some constant $C$ which is 
independent of $\eps$ but
may depend on $t$.

(ii) We use $(p,q)$ to denote the H\"older exponents $\tfrac{1}{p}+\tfrac{1}{q}=1$, and always choose $p\gg1$.

(iii) $G_t(x)=(2\pi t)^{-d/2}\exp(-|x|^2/2t)$ denotes the standard heat kernel.

(iv)
We let 
$H$ denote the Hilbert space $L^2(\R^{d+1})$, with norm
$\|\cdot\|_H$ and inner product $\la\cdot,\cdot\ra_H$.

(v) $\{B^j_t,W^j_t: t\geq0, j=1,\ldots\}$ is a family of standard independent $d-$dimensional Brownian motions built on another probability space $(\Sigma,\mathcal{A},\tilde{\Pb})$.  We will use $\E_B,\E_W,\Pb_B,\Pb_W$ when taking the expectation and the probability with respect to $B,W$ separately.

(vi) We use $d_{\mathrm{TV}}(\cdot,\cdot)$ to denote the total variation 
distance between two distributions, and if $X,Y$ are random variables of laws $\mu_X,\mu_Y$, we write
$d_{\mathrm{TV}}(X,Y)$ for $d_{\mathrm{TV}}(\mu_X,\mu_Y)$.

(vii) We let $\|\cdot \|_{\mathrm{op}}$ denote the operator norm.
\medskip

\noindent{\bf Acknowledgments.} We thank the two anonymous referees for a very careful reading
of the manuscript and many helpful suggestions to improve the presentation. AD was supported  by an NSF Graduate Research Fellowship, YG by NSF grant
DMS-1613301/1807748/1907928 and the Center for Nonlinear Analysis of CMU, LR by NSF grant DMS-1613603 and ONR grant N00014-17-1-2145, and
OZ by an Israel Science Foundation grant and  funding from the European Research Council (ERC) under the European Unions Horizon 2020 research and innovation programme (grant agreement No. 692452).

\section{Sketch of the proof}\label{s.sketch}

We rely on the Feynman-Kac representation of the solution to \eqref{e.maineq}:
\begin{equation}\label{e.fkre}
u(t,x)=\E_B\left[\exp\Big\{\beta \int_0^t V(t-s,x+B_s)ds -\frac12\beta^2R(0)t\Big\}\right],
\end{equation}
which has the same distribution, viewed as a random field in $x$, with $t$ fixed, as 
\begin{equation}\label{dec312}
Z(t,x)=\E_B[M(t,x)],
\end{equation}
with 
\begin{equation}\label{dec314}
M(t,x)=\exp\left(\beta\int_0^t V(s,x+B_s)ds-\frac12\beta^2R(0)t\right).
\end{equation}
We used the notation $Z(t,x)$ since it can be viewed as the partition function of a directed polymer of length $t$ and starting at $x$. For fixed $B$ and $x$, $M(\cdot,x)$ is a martingale. By \cite[Theorem 2.1]{mukherjee2016weak}, 
for $\beta$ small enough,
\begin{equation}\label{e.zinfinity}
\lim_{t\to\infty} Z(t,0)=Z_\infty
\end{equation}
almost surely, where $Z_\infty$ is a positive random variable satisfying $\E[Z_\infty]=1$.
Defining
\[
Z_\eps(t,x)=Z(\frac{t}{\eps^2},\frac{x}{\eps}), \     \ M_\eps(t,x)=M(\frac{t}{\eps^2},\frac{x}{\eps}),
\]
it suffices to consider the random variable
\[
X_\eps(t)=\int_{\R^d} \log Z_\eps(t,x)g(x)dx.
\]
The main result \eqref{e.mainre} is equivalent to 
\[
\frac{1}{\eps^{(d-2)/2}}(X_\eps(t)-\E[X_\eps(t)])\Rightarrow \beta\int_{\R^d}\U(t,x) g(x)dx.
\]
Throughout the paper, the temporal variable $t>0$ is fixed, so sometimes we will omit the dependence. Let us define
\begin{equation}\label{e.defsigma}
\begin{aligned}
\sigma_t^2&=\beta^2\Var\Big[\int_{\R^d} \U(t,x)g(x)dx\Big]\\
&=\beta^2 \Var\Big[ \int_{\R^d} g(x) \left(\nu_{\mathrm{eff}}\int_0^t\int_{\R^d} G_{t-s}(x-y)dW(s,y)\right)dx\Big]\\
&=\beta^2 \nu_{\mathrm{eff}}^2 \int_0^t \int_{\R^{2d}} g(x_1)g(x_2) G_{2s}(x_1-x_2)dx_1dx_2ds,
\end{aligned}
\end{equation}
where we recall 
$G_t(x)$ is the centered Gaussian density of variance $t$ in 
$\R^d$. The proof of Theorem~\ref{t.mainth} consists of two steps:

\begin{proposition}\label{p.convar}
Under the assumption of Theorem~\ref{t.mainth}, as $\eps\to0$, 
\[
\eps^{-(d-2)}\Var[X_\eps(t)]\to \sigma_t^2.
\]
\end{proposition}

\begin{proposition}\label{p.gauss}
Under the assumption of Theorem~\ref{t.mainth}, as $\eps\to0$,
\[
\frac{X_\eps(t)-\E[X_\eps(t)]}{\sqrt{\Var[X_\eps(t)]}}\Rightarrow N(0,1).
\]
\end{proposition}
To prove the convergence of the variance, we use the Clark-Ocone formula to write $X_\eps-\E[X_\eps]$ as a stochastic integral with respect to $\dot{W}$. An appropriate decomposition enables us to carry out some explicit calculations. To prove the Gaussianity, we use the second order Poincar\'e inequality \cite{chatterjee2009fluctuations,nourdin2009second} which involves estimating moments of Malliavin derivatives of $X_\eps$ and seems to be particularly handy in this context. 

In the rest of this section, we describe the main steps in the proof.

\subsection{Negative moments}

Throughout the paper, we rely on the existence of negative moments 
of $Z(t,x)$ for small $\beta$, as a quantitative control on the small ball
probability for $Z_\eps(t,x)$.

\begin{proposition}\label{p.negative}
There exits $\beta_0>0$ such that if $\beta<\beta_0$, 
\[
\sup_{t>0}\E[Z(t,x)^{-n}] \leq C_{\beta,n},
\]
for all $n\in{\mathbb N}$.
\end{proposition}

The proof is presented in Appendix~\ref{s.nemm}.

\subsection{The Clark-Ocone representation}

For each realization of the Brownian motion $B$, we can write 
\[
\begin{aligned}
\int_0^{t/\eps^2} V(s,\frac{x}{\eps}+B_s)ds=\int_0^{t/\eps^2} \left(\int_{\R^d} \varphi(\frac{x}{\eps}+B_s-y)\dot{W}(s,y)dy\right)ds
=\int_{\R^{d+1}} \Phi_{t,x,B}^\eps(s,y)dW(s,y),
\end{aligned}
\]
with 
\begin{equation}\label{e.defPhieps}
\Phi_{t,x,B}^\eps(s,y)=\1_{[0,t/\eps^2]}(s) \varphi(\frac{x}{\eps}+B_s-y).
\end{equation}
Therefore, we have
\[
\begin{aligned}
D_{s,y} Z_\eps(t,x)=D_{s,y} \E_B[M_\eps(t,x)]
= \beta \E_B\left[ M_\eps(t,x)\Phi_{t,x,B}^\eps(s,y)\right],
\end{aligned}
\]
where $D_{s,y}$ denotes the Malliavin derivative operator with respect to $\dot{W}$.\footnote{This involves an abuse of notation: we consider 
$D_{s,\cdot}Z_\eps(t,x)$
  as an element of the Hilbert space
  $H_1=L^2(\R^d)$, which is then integrated against the cylindrical 
  white noise $\dot{W}(s,\cdot)$. The Malliavin derivative at time $s$  is then an 
  element of $H_1$, which we write as $D_{s,y}Z_\eps(t,x)$. 
  See e.g. \cite{MZ} for background.}
  By Lemma~\ref{l.delog}, we have 
\[
D_{s,y}\log Z_\eps(t,x)= \frac{D_{s,y} Z_\eps(t,x)}{Z_\eps(t,x)},
\]
and the Clark-Ocone formula gives
\begin{equation} \label{eq-star1}
\begin{aligned}
X_\eps-\E[X_\eps]&=\int_{\R^{d+1}} \E[D_{s,y} X_\eps |\F_s] dW(s,y)
 =\int_{\R^{d+1}} \E\left[ \int_{\R^d} \frac{D_{s,y}Z_\eps(t,x)}{Z_\eps(t,x)}g(x)dx \bigg| \F_s\right]dW(s,y)\\
&=\beta\int_0^{t/\eps^2}\int_{\R^d}\left(\int_{\R^d} g(x)\E\left[ \frac{\E_B[ M_\eps(t,x)\Phi^\eps_{t,x,B}(s,y)]}{Z_\eps(t,x)} \bigg| \F_s\right]dx \right) dW(s,y).
\end{aligned}
\end{equation}
Here, $\F_s$ is the filtration generated by $\dot{W}(\ell, \cdot)$ up to $\ell\leq s$.

For 
\[
K=\eps^{-\alpha}
\]
 with some $\alpha>0$ to be determined, we decompose the  stochastic integral in \eqref{eq-star1}
into three parts:
\[
X_\eps-\E[X_\eps]=\beta(I_{1,\eps}+I_{2,\eps}+I_{3,\eps})
\]
with 
\begin{equation}\label{e.defI1}
I_{1,\eps}=\int_0^K\int_{\R^d}\left(\int_{\R^d} g(x)\E\left[ \frac{\E_B[ M_\eps(t,x)\Phi^\eps_{t,x,B}(s,y)]}{Z_\eps(t,x)} \bigg| \F_s\right]dx \right) dW(s,y),
\end{equation}
\begin{equation}\label{e.defI2}
I_{2,\eps}=\int_{K}^{t/\eps^2} \int_{\R^d}\left(\int_{\R^d} g(x)\E\left[ \frac{\E_B[ M_\eps(t,x)\Phi^\eps_{t,x,B}(s,y)]}{Z(K,x/\eps)}\left(\frac{Z(K,x/\eps)}{Z(t/\eps^2,x/\eps)}-1\right) \bigg| \F_s\right]dx \right) dW(s,y),
\end{equation}
and
\begin{equation}
  \label{e.defI3}
I_{3,\eps}=\int_{K}^{t/\eps^2} \int_{\R^d}\left(\int_{\R^d} g(x)\E\left[ \frac{\E_B[ M_\eps(t,x)\Phi^\eps_{t,x,B}(s,y)]}{Z(K,x/\eps)}\bigg| \F_s\right]dx \right) dW(s,y).
\end{equation}
The goal is to show that if $1\ll K\ll\eps^{-2}$, then
 the contribution from $I_{1,\eps},I_{2,\eps}$ is small compared to that from~$I_{3,\eps}$. For $I_{3,\eps}$, since the integration is in $s\geq K$, the random variable $Z(K,x/\eps)$ is $\F_s-$measurable, thus 
\[
I_{3,\eps}=\int_{K}^{t/\eps^2} \int_{\R^d}\left(\int_{\R^d} \frac{g(x)}{Z(K,x/\eps)}\E\left[ \E_B[ M_\eps(t,x)\Phi^\eps_{t,x,B}(s,y)]\bigg| \F_s\right]dx \right) dW(s,y).
\]
It turns out the inner conditional expectation can be computed explicitly, facilitating the analysis.

\subsection{The second order Poincar\'e inequality}

To simplify the notation, we define 
\[
Y_\eps=\frac{X_\eps-\E[X_\eps]}{\sqrt{\Var[X_\eps]}}.
\]
To show that 
$Y_\eps\Rightarrow N(0,1)$, we apply the second order Poincar\'e inequality \cite[Theorem 1.1]{nourdin2009second}. 
Since~$\E[Y_\eps]=0$ and $\Var[Y_\eps]=1$, 
with $\zeta$ a standard centered Gaussian random variable, we have
\begin{equation}\label{e.2ndpoincare}
d_{\mathrm{TV}}(Y_\eps,\zeta) \les \E[\|DY_\eps\|_H^4]^{1/4} \E[\|D^2Y_\eps\|_{\mathrm{op}}^4]^{1/4}.
\end{equation}
Recall our notation convention from 
Section \ref{sec-notation}; here,
$\|D^2Y_\eps \|_{\mathrm{op}}$ denotes the operator norm of the 
mapping $D^2Y_\eps: H\to H$ given by $h \mapsto  D^2 Y_\eps h$, that is,
\[
\|D^2Y_\eps\|_{\mathrm{op}}=\sup_{h,g\in H, \|h\|_H=\|g\|_H=1} \langle g, D^2Y_\eps h\rangle_H.
\]

It is clear that 
\[
DY_\eps=\frac{DX_\eps}{\sqrt{\Var[X_\eps]}}, \    \ D^2 Y_\eps=\frac{ D^2X_\eps}{\sqrt{\Var[X_\eps]}}.
\]
 Since $\Var[X_\eps]\sim \eps^{d-2}$ by Proposition~\ref{p.convar}, 
 to show $d_{\mathrm{TV}}(Y_\eps,\zeta)\to0$ using \eqref{e.2ndpoincare}, we only need to prove
\begin{equation}\label{e.sept41}
\E[\|DX_\eps\|_H^4]^{1/4} \E[\|D^2X_\eps\|_{\mathrm{op}}^4]^{1/4}=o(\eps^{d-2}),  \mbox{ as } \eps\to0.
\end{equation}

\section{Convergence of the variance}\label{s.convar}

Let us set 
\[
M_{\eps,j}(t,x)=\exp\left(\beta\int_0^{t/\eps^2}V(s,\frac{x}{\eps}+B_s^j)ds-\frac{\beta^2R(0)t}{2\eps^2}\right),
\]
where $B^j$ are independent Brownian motions. For any set $I\subset \R_+, x\in\R^d$ and two Brownian motions~$B^i,B^j$, we define 
\begin{equation}\label{e.defcR}
\cR(I,x,B^i,B^j)=\int_IR(x+B^i_s-B^j_s)ds
\end{equation}
as the weighted
intersection time of $B^i,B^j$ during the time interval $I$, with $x$ being the initial distance. For $I=[0,T]$, we simply write $\cR(T,x,B^i,B^j)$. Recall that $\Phi^\eps_{t,x,B}$ was defined in \eqref{e.defPhieps}, a straightforward calculation gives 
\[
\begin{aligned}
\int_{\R^{1+d}} \Phi^\eps_{t,x_1,B^1}(s,y)\Phi^\eps_{t,x_2,B^2}(s,y)dyds=&\int_0^{t/\eps^2} \left(\int_{\R^d} \varphi(\frac{x_1}{\eps}+B^1_s-y)\varphi(\frac{x_2}{\eps}+B^2_s-y)dy\right)ds\\
=&\cR(\frac{t}{\eps^2},\frac{x_1-x_2}{\eps},B^1,B^2).
\end{aligned}
\]

Before entering the proof, we present the following lemma which will be used repeatedly.

\begin{lemma}\label{l.holder}
For any $n\in\Z_+$ and $q>1$, there exists $\beta(n,q)>0$ such that if $\beta<\beta(n,q)$, then for any random variable $F(B^1,\ldots,B^n)\geq0$, $t>0$ and $\{x_j\in\R^d\}_{j=1,\ldots,n}$, we have
\begin{equation}\label{e.keyes}
\E\Big[ \frac{\E_B[ \prod_{j=1}^n M_{\eps,j}(t,x_j) F(B^1,\ldots,B^n)]}{\prod_{j=1}^nZ_\eps(t,x_j)}\Big] 
\les \E_B[ |F(B^1,\ldots,B^n)|^q]^{1/q}.
\end{equation}
\end{lemma}
\begin{proof}
By the
Cauchy-Schwarz inequality and Proposition~\ref{p.negative}, the square of the  l.h.s. of \eqref{e.keyes} is bounded by 
\[
\begin{aligned}
&\E\Big[\Big|\E_B\big[ \prod_{j=1}^n M_{\eps,j}(t,x_j) F(B^1,\ldots,B^n)\big]\Big|^2\Big]
=\E_B\E\Big[ \prod_{j=1}^{2n} M_{\eps,j}(t,x_j) F(B^1,\ldots,B^n) F(B^{n+1},\ldots,B^{2n})\Big],
\end{aligned}
\]
where $x_{j+n}=x_j$ for $j=1,\ldots,n$. Evaluating the expectation with respect to $\dot{W}$, we obtain 
\[
\E\Big[\prod_{j=1}^{2n} M_{\eps,j}(t,x_j)\Big]= \exp\Big(\frac{\beta^2}{2}
\sum_{j,k=1}^n\1_{j\neq k} \cR(\frac{t}{\eps^2},\frac{x_j-x_k}{\eps},B^j,B^k)\Big).
\]
Taking 
$p={q}/{(q-1)}$, Lemma~\ref{l.exmm} shows that the r.h.s. of the above expression has an $L^p$ norm that is bounded uniformly in $\eps,t$ and $x_j$, provided that $\beta$ is chosen small. We apply H\"older's
inequality to complete the proof.
\end{proof}

\subsection{The analysis of $I_{1,\eps}$}
Recall that the integral $I_{1,\eps}$ is given by \eqref{e.defI1}.
\begin{lemma}\label{l.i1}
If $K=\eps^{-\alpha}$ with $\alpha<{4}/{(2+d)}$, we have 
\[
\eps^{-(d-2)}\E[I_{1,\eps}^2]\to 0,~~\hbox{ as $\eps\to0$}.
\] 
\end{lemma}

\begin{proof}
Writing
\[
I_{1,\eps}
=\int_0^K\int_{\R^d} \E[ \Y_{s,y}|\F_s] dW(s,y),
\]
for the  appropriate
$\Y_{s,y}$ as in \eqref{e.defI1}, we have by
It\^o's isometry that
\[
\begin{aligned}
\E[I_{1,\eps}^2]&=\int_0^K\int_{\R^d} \E[ |\E[ \Y_{s,y}|\F_s] |^2]dyds \leq \int_0^K\int_{\R^d} \E[\Y_{s,y}^2] dyds
\\
&=\beta^2\int_0^K\int_{\R^d} \int_{\R^{2d}} g(x_1)g(x_2) \E\Big[ \frac{\E_B[\prod_{j=1}^2 M_{\eps,j}(t,x_j)
\Phi^\eps_{t,x_j,B^j}(s,y)]}{Z_\eps(t,x_1)Z_\eps(t,x_2)} \Big]dx_1dx_2dyds.
\end{aligned}
\]
Using the fact that 
\[
\int_0^K\int_{\R^d} \prod_{j=1}^2\Phi^\eps_{t,x_j,B^j}(s,y)dyds=\int_0^K  R(\frac{x_1-x_2}{\eps}+B^1_s-B^2_s)ds=
\cR(K,\frac{x_1-x_2}{\eps},B^1,B^2),
\]
where we recall that 
$R$ is the spatial covariance function defined in \eqref{e.defR}, we have 
\[
\E[I_{1,\eps}^2] \les \int_{\R^{2d}} g(x_1)g(x_2) \E\Big[ \frac{\E_B[\prod_{j=1}^2 M_{\eps,j}(t,x_j)
\cR(K,\frac{x_1-x_2}{\eps},B^1,B^2)]}{Z_\eps(t,x_1)Z_\eps(t,x_2)} \Big]dx_1dx_2.
\]
By Lemma~\ref{l.holder}, we have
\[
\E[I_{1,\eps}^2]\les \int_{\R^{2d}} g(x_1)g(x_2) \sqrt{\E_B[\cR(K,\tfrac{x_1-x_2}{\eps},B^1,B^2)^2]}dx_1dx_2.
\]
By the expression of $\cR$ in \eqref{e.defcR}, it suffices to use the estimate
\[
\begin{aligned}
\E_B[\cR(K,\tfrac{x_1-x_2}{\eps},B^1,B^2)^2] \les& K^2 \Pb_B\left[ \max_{s\in[0,K]}|B^1_s-B^2_s| \geq \tfrac{|x_1-x_2|}{C\eps}\right]
\les K^2 e^{-\frac{|x_1-x_2|^2}{C\eps^2K}}
\end{aligned}
\]
for some constant $C>0$. This implies 
\[
\E[I_{1,\eps}^2] \les K \int_{\R^{2d}}g(x_1)g(x_2)  e^{-\frac{|x_1-x_2|^2}{C\eps^2K}} dx_1dx_2 \les K^{1+\frac{d}{2}} \eps^d.
\]
The proof is complete.
\end{proof}

\subsection{The analysis of $I_{2,\eps}$}

Recall the definition of $I_{2,\eps}$, see \eqref{e.defI2}, and set $K=\eps^{-\alpha}$.
We have the following lemma.
\begin{lemma}\label{l.i2}
For any $\alpha>0$, there exists $\beta(\alpha)$ such that if $\beta<\beta(\alpha)$, 
\[
\eps^{-(d-2)} \E[I_{2,\eps}^2]\to0,~~\hbox{ as $\eps\to0$.}
\]
\end{lemma}

\begin{proof}
By the same calculation as in the proof of Lemma~\ref{l.i1}, we have
\[
\E[I_{2,\eps}^2] \les \int_{\R^{2d}} g(x_1)g(x_2)A_\eps(x_1,x_2)dx_1dx_2,
\]
with 
\[
A_\eps(x_1,x_2)=\E\Big[\frac{\E_B[\prod_{j=1}^2 M_{\eps,j}(t,x_j)\cR([K,t/\eps^2],
{(x_1-x_2)}/{\eps},B^1,B^2)] }{Z(K,x_1/\eps)Z(K,x_2/\eps)} 
\prod_{j=1}^2\Big(\frac{Z(K,x_j/\eps)}{Z(t/\eps^2,x_j/\eps)}-1\Big)\Big].
\]
Recall that $\cR([K,t/\eps^2],
{(x_1-x_2)}/{\eps},B^1,B^2)$ was defined in \eqref{e.defcR} and measures the intersection time of $B^1,B^2$ during $[K,t/\eps^2]$.
Applying Proposition~\ref{p.negative}, H\"older's inequality,
and the fact that $Z(t,x)$ is stationary in $x$, we have 
\begin{equation}\label{e.se9191}
\begin{aligned}
A_\eps(x_1,x_2)\les&  
\Big(\E[|\E_B[\prod_{j=1}^2M_{\eps,j}(t,x_j)\cR([K,t/\eps^2],{(x_1-x_2)}/{\eps},B^1,B^2)]|^2]\Big)^{1/2}
\\
&\times 
\left(\E[|Z(K,0)-Z(t/\eps^2,0)|^{16}]\right)^{1/8}.
\end{aligned}
\end{equation}
For the second factor on the r.h.s. of \eqref{e.se9191}, we have
\begin{equation}\label{e.se9192}
\begin{aligned}
\E[|Z(K,0)-Z(t/\eps^2,0)|^{16}] \les& \sqrt{\E[|Z(K,0)-Z(t/\eps^2,0)|^2]\sup_{t>0} \E[|Z(t,0)|^{30}]}\\
\les & K^{-(d-2)/4}=\eps^{\alpha(d-2)/4},
\end{aligned}
\end{equation}
where the second ``$\les$'' comes from e.g. \cite[Proposition 2.1]{CometsMukh}. 
For first factor on the r.h.s. of \eqref{e.se9191}, the same calculation as in the proof of Lemma~\ref{l.i1} yields
\begin{equation}\label{e.se9193}
\begin{aligned}
&\E[|\E_B[\prod_{j=1}^2 M_{\eps,j}(t,x_j)\cR([K,t/\eps^2],{(x_1-x_2)}/{\eps},B^1,B^2)]|^2]\\
&\les\E_B[|\cR([K,t/\eps^2],{(x_1-x_2)}/{\eps},B^1,B^2)|^q]^{2/q}\les 
\frac{1}{|x_1-x_2|^{2(d-2)/q}}\eps^{\frac{2(d-2)}{q}-\frac{4}{p}},
\end{aligned}
\end{equation}
where the last step comes from Lemma~\ref{l.qmm}
below. Combining \eqref{e.se9192} and \eqref{e.se9193}, we have 
\[
A_\eps(x_1,x_2) \les \frac{\eps^{\frac{(d-2)}{q}-\frac{2}{p}}}{|x_1-x_2|^{(d-2)/q}}\times \eps^{\frac{\alpha(d-2)}{32}},
\]
which implies 
\[
\E[I_{2,\eps}^2] \les \eps^{\frac{(d-2)}{q}-\frac{2}{p}+\frac{\alpha(d-2)}{32}}.
\]
We choose $p$ large enough (for fixed 
$\alpha$)
such that 
\[
\frac{(d-2)}{q}-\frac{2}{p}+\frac{\alpha(d-2)}{32}
=\frac{d}{q}-2+\frac{\alpha(d-2)}{32}>d-2,
\]
to complete the proof.
\end{proof}

\begin{lemma}\label{l.qmm}
For any $q>1$, we have
\[
\E_B[|\cR(t/\eps^2,x/\eps,B^1,B^2)|^q]\les \frac{\eps^{d-2-\frac{2q}{p}}}{|x|^{d-2}}.
\]
\end{lemma}

\begin{proof}
Since $R(x)=0$ for $|x|\geq1$, we have 
\[
\begin{aligned}
|\cR(t/\eps^2,x/\eps,B^1,B^2)|^q \les& 
\Big(\int_0^{t/\eps^2} \1_{|x/\eps+B^1_s-B^2_s|\leq 1} ds\Big)^q
\les \eps^{-2q/p}\int_0^{t/\eps^2} \1_{|x/\eps+B^1_s-B^2_s|\leq 1} ds.
\end{aligned}
\]
Taking the expectation, we obtain 
\[
\begin{aligned}
\E[|\cR(t/\eps^2,x/\eps,B^1,B^2)|^q ] \les& \eps^{-2q/p}\int_0^\infty \E_B[ \1_{|x/\eps+B^1_s-B^2_s|\leq 1} ]ds
\les  \eps^{-2q/p}\int_{\R^d}  \frac{\1_{|x/\eps+y|\leq 1}}{|y|^{d-2}}dy \les \frac{\eps^{d-2-\frac{2q}{p}}}{|x|^{d-2}},
\end{aligned}
\]
which completes the proof.
\end{proof}

\subsection{The analysis of $I_{3,\eps}$}\label{s.analysisi3}
Recall the definition of $I_{3,\eps}$, see \eqref{e.defI3}.
Using the fact that $\E[M(t/\eps^2,x/\eps)|\F_s]=M(s,x/\eps)$, we get
\[
I_{3,\eps}=\int_K^{t/\eps^2}\int_{\R^d} \left(\int_{\R^d} \frac{g(x)}{Z(K,x/\eps)}\E_B[ M(s,x/\eps)\Phi^\eps_{t,x,B}(s,y)] dx\right) dW(s,y).
\]
For any $T>0,x_1,x_2\in\R^d$ and a standard $d$-dimensional 
Brownian motion $\bar{B}$, we define 
\[
\cH(T,x_1,x_2)=\E_B\left[ e^{\beta^2 \int_0^T R(x_1+\bar{B}_{2s})ds} \big|\bar{B}_{2T}=x_2\right].
\]
We introduce the following notation: for any $x,y\in\R^d$, the expectation $\hat{\E}_{x,y}$ is defined as 
\begin{equation}\label{e.defhatE}
\hat{\E}_{x,y}[F]=\E\left[\frac{\E_B[ M_1(K,x)M_2(K,y) F]}{Z(K,x)Z(K,y)}\right]
\end{equation}
for any random variable $F$, where we recall that $M_j$ is the $M$ associated with $B^j$. In particular, we will consider functionals of 
\[
\X_K=B^1_K-B^2_K,
\]
so
\[
\hat{\E}_{x,y}[F(\X_K)]=\E\left[\frac{\E_B[ M_1(K,x)M_2(K,y) F(B^1_K-B^2_K)]}{Z(K,x)Z(K,y)}\right].
\]
The following three lemmas combine to show the convergence of 
\begin{equation}\label{e.coni3}
\eps^{-(d-2)}\E[I_{3,\eps}^2]\to \beta^{-2}\sigma_t^2=\nu_{\mathrm{eff}}^2 \int_0^t \int_{\R^{2d}} g(x_1)g(x_2) G_{2s}(x_1-x_2)dx_1dx_2ds.
\end{equation} 

\begin{lemma}\label{l.i3re}
With $K=\eps^{-\alpha}$ and $\alpha<2$, we have
\[
\eps^{-(d-2)}\E[I_{3,\eps}^2]= \int_{0}^{t-\eps^2K} \G_\eps(s) ds,
\]
with 
\begin{equation}
\begin{aligned}
\G_\eps(s)=\int_{\R^{3d}}& g(x-w)g(x) R(y)  
\hat{\E}_{-w/\eps,0}\left[G_{2s}(w+\eps y-\eps \X_K)\cH(\frac{s}{\eps^2},y,\X_K-\frac{w}{\eps}-y)\right]dxdydw.
\end{aligned}
\end{equation}
\end{lemma}

\begin{lemma}\label{l.i3bd}
There exists $\beta_0>0$ so that there exists $\gamma\in (0,1)$ such that,  for all  $\beta<\beta_0$,
$\G_\eps(s)\les s^{-\gamma}$ for $s\in(0,t)$.
\end{lemma}

\begin{lemma}\label{l.i3lim}
For any $s\in(0,t)$, 
\[
\G_\eps(s)\to \nu_{\mathrm{eff}}^2 \int_{\R^{2d}} g(x-w)g(x)G_{2s}(w)dwdx,
\]
 as $\eps\to0$.
\end{lemma}

\begin{proof}[Proof of Lemma~\ref{l.i3re}]
By It\^o's isometry, we have 
\[
\E[I_{3,\eps}^2]=\int_K^{t/\eps^2} \int_{\R^{3d}} g(x_1)g(x_2)\E\Big[
\frac{\E_B[ \prod_{j=1}^2 M_j(s,x_j/\eps)\Phi^\eps_{t,x_j,B^j}(s,y)]}{Z(K,x_1/\eps)Z(K,x_2/\eps)}\Big] dydx_1dx_2ds.
\]
Conditioning on $\F_K$, we have
\[
\E\Big[ \prod_{j=1}^2 M_j(s,x_j/\eps)\,\big|\,\F_K\Big]=
\Big(\prod_{j=1}^2 M_j(K,x_j/\eps)\Big)\times 
\exp\left(\beta^2 \int_K^sR(\frac{x_1-x_2}{\eps}+B^1_r-B^2_r)dr\right).
\]
Integrating in $y$, we have 
\begin{equation}\label{e.var1}
\begin{aligned}
\E[I_{3,\eps}^2]= &\int_{K}^{t/\eps^2}\int_{\R^{2d}} g(x_1)g(x_2)\\
&
\!\!\!\!\!
\!\!\!\!\!
\!\!\!\!\!
\!\!\!\!\!
\!\!\!\!\!
\!\!\!\!\!
\!\!\!\!\!
\times  \E\Big[ \frac{\E_B[ \prod_{j=1}^2 M_j(K,x_j/\eps) 
\exp\{\beta^2 \int_K^sR(\frac{x_1-x_2}{\eps}+B^1_r-B^2_r)dr\} R(\frac{x_1-x_2}{\eps}+B^1_s-B^2_s)]}{Z(K,x_1/\eps)Z(K,x_2/\eps)}\Big]dx_1dx_2ds.
\end{aligned}
\end{equation}
Changing variables $x_2\mapsto x, x_1\mapsto x+\eps y$ and $s\mapsto \tfrac{s}{\eps^2}$, and using the stationarity of $V$ in $x$, we have 
\[
\begin{aligned}
\eps^{-(d-2)}\E[I_{3,\eps}^2]=& \int_{\eps^2K}^t  \int_{\R^{2d}} g(x+\eps y)g(x)\\
&
\!\!\!\!\!
\!\!\!\!\!
\!\!\!\!\!
\!\!\!\!\!
\!\!\!\!\!
\!\!\!\!\!
\!\!\!\!\!
\times \E\Big[ \frac{\E_B[M_1(K,y)M_2(K,0)\exp\{\beta^2 
\int_K^{s/\eps^2}R(y+B^1_r-B^2_r)dr\} R(y+B^1_{s/\eps^2}-B^2_{s/\eps^2})]}{Z(K,y)Z(K,0)}\Big]dxdyds.
\end{aligned}
\]
For $r\geq K$, we can write $B^1_r-B^2_r=B^1_K-B^2_K+\bar 
B_{2(r-K)}$, where $\bar B$ 
is another Brownian motion independent of $B$ and the random environment $V$. Thus, recall the definition of $\hat{\E}$ in \eqref{e.defhatE}, we have
\[
\begin{aligned}
\eps^{-(d-2)}\E[I_{3,\eps}^2]=& \int_{\eps^2K}^t  \int_{\R^{2d}} g(x+\eps y)g(x)\\
&\times\hat{\E}_{y,0}\E_{\bar B}
\Big[\exp\{\beta^2 \int_K^{s/\eps^2}R(y+\X_K+\bar B_{2(r-K)})dr\} R(y+\X_K+\bar B_{2(s/\eps^2-K)})\Big]dxdyds.
\end{aligned}
\]
We write the expectation with respect to $\bar B$ more explicitly by conditioning on the end point of~$\bar B$:
\[
\begin{aligned}
  &\E_{\bar B}\left[ e^{\beta^2 \int_K^{s/\eps^2}R(y+\X_K+\bar B_{2(r-K)})dr} R(y+\X_K+\bar B_{2(s/\eps^2-K)})\right]\\
  &= \int_{\R^d} G_{2(s-\eps^2K)}(w) R(y+\X_K+w/\eps)
  \E_{\bar B}\Big[ e^{\beta^2 \int_K^{s/\eps^2}R(y+\X_K+\bar B_{2(r-K)})dr}  
  \big|\bar B_{2(s/\eps^2-K)}=w/\eps\Big] dw.
\end{aligned}
\]
Now we consider the integral in $y$, change variable $y\mapsto y-{w}/{\eps}$ and use the time-reversal of the Brownian bridge, then the expectation in $\bar B$ in the last display
becomes 
\[
 \E_{\bar B}
 \Big[ e^{\beta^2 \int_K^{s/\eps^2}R(y+\X_K-w/\eps+\bar B_{2(r-K)})dr}  
 \big|\bar B_{2(s/\eps^2-K)}=\frac{w}\eps\Big]
= \E_{\bar B}\Big[ e^{\beta^2 \int_0^{s/\eps^2-K} 
R(y+\X_K+\bar B_{2r})dr}\big|\bar B_{2(s/\eps^2-K)}=-\frac{w}\eps\Big],
\]
and we can write 
\[
\begin{aligned}
\eps^{-(d-2)}\E[I_{3,\eps}^2]=& \int_{\eps^2K}^t \int_{\R^{3d}} g(x+\eps y-w)g(x) G_{2(s-\eps^2K)}(w)
\\ &\times 
\hat{\E}_{y-{w}/{\eps},0}[R(y+\X_K) \cH(s/\eps^2-K,y+\X_K,-w/\eps)] dxdydwds.
\end{aligned}
\]
We change back  variables in the form   $w\mapsto w+\eps y$ to obtain 
\[
\begin{aligned}
\eps^{-(d-2)}\E[I_{3,\eps}^2]=& \int_{\eps^2K}^t \int_{\R^{3d}} g(x-w)g(x) G_{2(s-\eps^2K)}(w+\eps y)\\
&\times \hat{\E}_{-w/\eps,0}[ R(y+\X_K) \cH(s/\eps^2-K, y+\X_K,-w/\eps-y)] dxdydwds.
\end{aligned}
\]
Finally, we change variables $y\mapsto y-\X_K$ and $s\mapsto s+\eps^2K$ to complete the proof.
\end{proof}

\begin{proof}[Proof of Lemma~\ref{l.i3bd}]
By Lemma~\ref{l.exmm}, we know that $\cH$ is uniformly bounded for small $\beta$, so that
\[
\G_\eps(s) \les \int_{\R^{3d}}|g(x-w)g(x)|\,R(y) \hat{\E}_{-w/\eps,0}[G_{2s}(w+\eps y-\eps \X_K)]dxdydw.
\]
We bound the expectation by Lemma~\ref{l.holder}: for any $q>1$, if $\beta<\beta(q)$ then 
\[
\begin{aligned}
 \hat{\E}_{-w/\eps,0}[G_{2s}(w+\eps y-\eps \X_K)]=&\E\left[\frac{\E_B[ M_1(K,-w/\eps)M_2(K,0)G_{2s}(w+\eps y-\eps \X_K)]}{Z(K,-w/\eps)Z(K,0)}\right]\\
 \les &\E_B[|G_{2s}(w+\eps y-\eps \X_K)|^{q}]^{1/q}.
\end{aligned}
\]
Recaling that $\X_K=B^1_K-B^2_K\sim N(0,2K)$, the above expectation can be computed explicitly:
\[
\begin{aligned}
\E_B[|G_{2s}(w+\eps y-\eps \X_K)|^{q}]^{1/q}\les \left(\frac{1}{s^{d(q-1)/2}} G_{\frac{2s}{q}+2\eps^2K}(w+\eps y)\right)^{1/q}.
\end{aligned}
\]
Thus 
\begin{equation}\label{dec302}
\G_\eps(s)\les \int_{\R^{3d}} |g(x)| \,R(y)  \left(\frac{1}{s^{d(q-1)/2}} G_{\frac{2s}{q}+2\eps^2K}(w)\right)^{1/q} dxdydw \les s^{-\frac{d}{2p}},
\end{equation}
where $(p,q)$ are the dual 
H\"older exponents and we used the fact that $s<t$ 
hence $2s/q+2\eps^2K\les1$ in the last step. We choose $p>d/2$ and adjust 
$\beta$ accordingly to complete the proof.
\end{proof}

\begin{proof}[Proof of Lemma~\ref{l.i3lim}]
Recall that 
\[
\begin{aligned}
\G_\eps(s)=\int_{\R^{3d}} &g(x-w)g(x) R(y)
\hat{\E}_{-w/\eps,0}\left[G_{2s}(w+\eps y-\eps \X_K)\cH(\frac{s}{\eps^2},y,\X_K-\frac{w}{\eps}-y)\right]dxdydw.
\end{aligned}
\]
Since $s>0$ is fixed, the expectation in the above expression is bounded uniformly in $x,y,w$, so we only need to pass to the limit of the expectation for fixed $x,y,w\in\R^d$ and $w\neq0$. The proof is divided into three steps.

(i) We show that $\hat{\E}_{-w/\eps,0}[|G_{2s}(w+\eps y-\eps \X_K)-G_{2s}(w)|]\to0$ as $\eps\to0$. Using the fact that 
\[
|G_{2s}(w+\eps y-\eps \X_K)-G_{2s}(w)|\les \eps|y|+\eps |\X_K|,
\]
it suffices to show $\hat{\E}_{-w/\eps,0}[|\eps \X_K|]\to0$. 
We apply Lemma~\ref{l.holder} to get
\begin{equation}\label{e.se202}
\hat{\E}_{-w/\eps,0}[|\eps \X_K|] \les \sqrt{ \E_B[|\eps \X_K|^2]}=\sqrt{2\eps^2K}\to0.
\end{equation}

(ii) For $\alpha\in(0,2)$, define 
\[
\tilde{\cH}_\eps=\E_{\bar B}\Big[ \exp\Big\{\beta^2 \int_0^{s/\eps^\alpha} 
R(y+{\bar B}_{2r})dr\Big\}\Big|\bar B_{2s/\eps^2}=\X_K-\frac{w}{\eps}-y\Big],
\]
we show that 
\begin{equation}\label{e.se201}
  \hat{\E}_{-w/\eps,0}[\tilde{\cH}_\eps]\to 
  \E_{\bar B}\left[\exp\Big\{\beta^2 \int_0^\infty R(y+\bar B_{2s})ds\Big\}\right]
\end{equation}
 as $\eps\to0$. Note that $\tilde{\cH}$ can be written more explicitly by conditioning on $\bar B_{2s/\eps^\alpha}$:
\[
\begin{aligned}
\tilde{\cH}_\eps=&\E_{\bar B}\Big[\exp\Big\{\beta^2\int_0^{s/\eps^\alpha}R(y+\bar B_{2r})dr
\Big\} \times  \frac{G_{2s(1-\eps^{2-\alpha})}(\eps \X_K-w-\eps y-\eps \bar B_{2s/\eps^\alpha})}{G_{2s}(\eps \X_K-w-\eps y)}\Big]\\
=&\frac{1}{(1-\eps^{2-\alpha})^{d/2}} 
\E_{\bar B}\Big[ \exp\Big\{\beta^2\int_0^{s/\eps^\alpha}R(y+\bar B_{2r})dr\Big\}\\ 
&\times
\exp\Big\{-\frac{(\eps \X_K-w-\eps y-\eps \bar B_{2s/\eps^\alpha})^2}
{4s(1-\eps^{2-\alpha})}+\frac{(\eps \X_K-w-\eps y)^2}{4s}\Big\}\Big].
\end{aligned}
\]
There are three factors inside the above expectation. First, we note that again
by an application of Lemma~\ref{l.holder}, we have
\[
\hat{\E}_{-w/\eps,0}[e^{\lambda |\eps \X_K|^2}]\les1,
\]
for any $\lambda>0$. Then by the same proof as
for (i), we can replace the second factor by $e^{-w^2/4s}$ with a negligible error. Finally we use the simple inequality $|e^x-e^y|\leq (e^x+e^y)|x-y|$ to replace the third factor by $e^{w^2/4s}$ with a negligible error. This proves \eqref{e.se201}.

(iii) We show that 
\begin{equation}\label{e.se203}
\hat{\E}_{-w/\eps,0}[|\cH(\frac{s}{\eps^2},y,\X_K-\frac{w}{\eps}-y)-\tilde{\cH}_\eps|]\to0
\end{equation}
as $\eps\to0$. We decompose the expectation into two parts according to the value of $\X_K$:
\[
\hat{\E}_{-w/\eps,0}[|\cH(\frac{s}{\eps^2},y,\X_K-\frac{w}{\eps}-y)-\tilde{\cH}_\eps|\,\1_{|\eps \X_K|>w/2}]+\hat{\E}_{-w/\eps,0}[|\cH(\frac{s}{\eps^2},y,\X_K-\frac{w}{\eps}-y)-\tilde{\cH}_\eps|\,\1_{|\eps \X_K|\leq w/2}].
\]
For the first term, applying Lemma~\ref{l.exmm} and \eqref{e.se202} yields 
\[
\hat{\E}_{-w/\eps,0}[|\cH(\frac{s}{\eps^2},y,\X_K-\frac{w}{\eps}-y)-\tilde{\cH}_\eps|\, \1_{|\eps \X_K|>w/2}]\les \hat{\E}_{-w/\eps,0}[\1_{|\eps \X_K|>w/2}]\to0.
\]
 For the second term, we have,
again using Lemma~\ref{l.exmm}, that
\begin{equation}\label{dec306}
\begin{aligned}
&|\cH(\frac{s}{\eps^2},y,\X_K-\frac{w}{\eps}-y)-\tilde{\cH}_\eps|\\
&=\E_{\bar B}\Big[ \exp\Big\{\beta^2\int_0^{s/\eps^2} R(y+\bar B_{2r})dr\Big\}
-\exp\Big\{\beta^2 \int_0^{s/\eps^\alpha} R(y+\bar B_{2r})dr\Big\} 
\,\Big|\,\bar B_{2s/\eps^2}=\X_K-\frac{w}{\eps}-y\Big]\\
&\leq \beta^2 \E_{\bar B}\Big[ \exp\Big\{\beta^2\int_0^{s/\eps^2} R(y+\bar B_{2r})dr\Big\}
\int_{s/\eps^\alpha}^{s/\eps^2}R(y+\bar B_{2r})dr
\,\Big|\,
\bar B_{2s/\eps^2}=\X_K-\frac{w}{\eps}-y\Big]\\
&\les \sqrt{\E_{\bar B}\Big[ \big(\int_{s/\eps^\alpha}^{s/\eps^2}R(y+\bar B_{2r})dr\big)^2 
\,\Big|\,\bar B_{2s/\eps^2}=\X_K-\frac{w}{\eps}-y\Big]}.
\end{aligned}
\end{equation}
By the condition of $|\eps\X_K|<w/2$ and $w\neq0$, we have $\eps \X_K-w-\eps y$ is away from the origin for small $\eps$. An application of Lemma~\ref{l.bridgetail} shows the above term goes to zero, uniformly in $|\eps \X_K|<w/2$. This completes the proof of \eqref{e.se203}.

To summarize, we have 
\[
\begin{aligned}
  \G_\eps(s)\to &\int_{\R^{3d}} g(x-w)g(x)R(y)G_{2s}(w)  
  \E_{\bar B}\Big[\exp\Big\{\beta^2 \int_0^\infty R(y+\bar B_{2s})ds\Big\}\Big] dxdydw\\
&=\nu_{\mathrm{eff}}^2 \int_{\R^{2d}} g(x-w)g(x)G_{2s}(w)dxdw,
\end{aligned}
\]
which completes the proof.
\end{proof}

\subsection{Proof of Proposition~\ref{p.convar}}

Recall that $X_\eps-\E[X_\eps]=$ $\beta(I_{1,\eps}+I_{2,\eps}+I_{3,\eps})$. 
Choosing $K=\eps^{-\alpha}$ with $\alpha\in(0,\frac{4}{2+d})$ and $\beta$ small, we combine Lemmas~\ref{l.i1}, \ref{l.i2} 
and \eqref{e.coni3} to obtain
\[
\eps^{-(d-2)}\Var[X_\eps]\to \sigma_t^2.
\]
\begin{remark}\label{r.kpz-she}
A simpler version of the proof will show the convergence of 
\[
\eps^{-(d-2)}\Var\left[\int_{\R^d} Z_\eps(t,x)g(x)dx\right]\to \sigma_t^2,
\]
that is, the convergence of the variance for the solution of 
the stochastic heat equation itself. 
A key identity in the proof of Lemma~\ref{l.i3re} is \eqref{e.var1}:
\[
\begin{aligned}
\E[I_{3,\eps}^2]= &\int_{K}^{t/\eps^2}\int_{\R^{2d}} g(x_1)g(x_2)\\
&\times  \hat{\E}_{\frac{x_1}{\eps},\frac{x_2}{\eps}}\E_B\Big[ 
\exp\Big\{\beta^2 \int_K^sR(\frac{x_1-x_2}{\eps}+B^1_r-B^2_r)dr\Big\} 
R(\frac{x_1-x_2}{\eps}+B^1_s-B^2_s)\Big]dx_1dx_2ds.
\end{aligned}
\]
For the stochastic heat equation, it is straightforward to check that
the above term becomes 
\[
\begin{aligned}
& \int_{K}^{t/\eps^2}\int_{\R^{2d}} g(x_1)g(x_2) 
\E_B\Big[\exp\Big\{\beta^2 \int_K^sR(\frac{x_1-x_2}{\eps}+B^1_r-B^2_r)dr\Big\} 
R(\frac{x_1-x_2}{\eps}+B^1_s-B^2_s)\Big]dx_1dx_2ds.
\end{aligned}
\]
The difference between these expressions
comes from the distribution of $(B^1_K,B^2_K)$. In the case of~KPZ, $B^j_K$ are distributed according to the polymer measure; in the case of SHE, $B^j_K$ is not weighed by the environment hence has distribution $N(0,K)$. It is clear from \eqref{e.se202} that the asymptotic behaviors of $\eps B^j_K$ are the same in two cases, which leads to the same limiting variances.
\end{remark}

\section{Gaussianity}\label{s.gauss}
We turn to the proof of Proposition~\ref{p.gauss}. The second order Poincar\'e
inequality (\ref{e.2ndpoincare}) reduces our task to showing that 
\[
\E[\|DX_\eps\|_H^4]^{1/4} \E[\|D^2X_\eps\|_{\mathrm{op}}^4]^{1/4}=o(\eps^{d-2}),  \mbox{ as } \eps\to0.
\]
Since 
\[
DX_\eps= \int_{\R^d} \frac{DZ_\eps(t,x)}{Z_\eps(t,x)}g(x)dx,
\]
we have
\[
\begin{aligned}
D^2X_\eps=&D\int_{\R^d} \frac{DZ_\eps(t,x)}{Z_\eps(t,x)} g(x)dx
=\int_{\R^d} \frac{ Z_\eps (t,x)D^2Z_\eps(t,x)-DZ_\eps(t,x)\otimes DZ_\eps(t,x)}{Z^2_\eps(t,x)}g(x)dx.
\end{aligned}
\]
Using the Feynman-Kac representation (\ref{dec312})-(\ref{dec314}) and the definition of $\Phi^\eps_{t,x,B}$ in \eqref{e.defPhieps} gives
\[
D^2Z_\eps(t,x)=\beta^2 \E_B[ M_\eps(t,x)\Phi_{t,x,B}^\eps\otimes \Phi_{t,x,B}^\eps],
\]
so that
\[
Z_\eps(t,x)D^2Z_\eps(t,x)=\beta^2 \E_{B}\Big[
\prod_{j=1}^2M_{\eps,j}(t,x)\Phi_{t,x,B^2}^\eps\otimes \Phi_{t,x,B^2}^\eps\Big],
\]
and
\[
DZ_\eps(t,x)\otimes D Z_\eps(t,x)=\beta^2 \E_{B}\Big[
\prod_{j=1}^2M_{\eps,j}(t,x)\Phi_{t,x,B^1}^\eps\otimes \Phi_{t,x,B^2}^\eps\Big].
\]
Thus, we can write 
\[
\begin{aligned}
D^2X_\eps&=\beta^2\int_{\R^d}  \frac{ \E_{B}[\prod_{j=1}^2M_{\eps,j}(t,x)(\Phi_{t,x,B^2}^\eps-\Phi_{t,x,B^1}^\eps)\otimes \Phi_{t,x,B^2}^\eps]}{Z_\eps^2(t,x)} g(x)dx
=\cP_2-\cP_1,
\end{aligned}
\]
where 
\[
\cP_k=\beta^2\int_{\R^d}  \frac{ \E_{B}[\prod_{j=1}^2M_{\eps,j}(t,x)\Phi_{t,x,B^k}^\eps
\otimes \Phi_{t,x,B^2}^\eps]}{Z^2_\eps(t,x)} g(x)dx\in H\otimes H.
\]
Recall that $\Phi_{t,x,B}^\eps$, 
defined in \eqref{e.defPhieps}, is an element of $H=L^2(\R^{d+1})$ for each $(\eps,t,x,B)$ fixed.
Thus, we have
\[
\|D^2X_\eps\|_{\mathrm{op}}^4\les \|\cP_1\|_{\mathrm{op}}^4+\|\cP_2\|_{\mathrm{op}}^4,
\]
and we only need to estimate $\E[\|\cP_k\|_{\mathrm{op}}^4]$. 

\subsection{The first derivative}
The goal of this section is to show the following lemma.
\begin{lemma}\label{l.1stde}
For any $\delta>0$, there exists $\beta(\delta)>0$ such that if $\beta<\beta(\delta)$,
then 
\[
\E[\|DX_\eps\|_H^4]^{1/4} \les \eps^{\frac{d-2}{2}-\delta}.
\]
\end{lemma}
\begin{proof}
A direct calculation gives
\[
\|DX_\eps\|_H^4
=\beta^4\int_{\R^{4d}} \prod_{j=1}^4 \frac{g(x_j)}{Z_\eps(t,x_j)}  
\E_{B}\Big[ \prod_{j=1}^4 M_{\eps,j}(t,x_j) 
\cR(\frac{t}{\eps^2},\frac{x_1-x_2}{\eps},B^1,B^2)
\cR(\frac{t}{\eps^2},\frac{x_3-x_4}{\eps},B^3,B^4)\Big]dx,
\]
with $\cR$ defined in \eqref{e.defcR}. Taking the expectation $\E$ and applying 
Lemma~\ref{l.holder}, we have
\[
\begin{aligned}
\E[\|DX_\eps\|_H^4] \les& \int_{\R^{4d}} \prod_{j=1}^4 |g(x_j)| \E_B\left[\cR^q(\frac{t}{\eps^2},\frac{x_1-x_2}{\eps},B^1,B^2)\cR^q(\frac{t}{\eps^2},\frac{x_3-x_4}{\eps},B^3,B^4)\right]^{1/q}\\
\les & \eps^{\frac{2(d-2)}{q}-\frac{4}{p}}\int_{\R^{4d}} \prod_{j=1}^4 |g(x_j)|  \frac{1}{|x_1-x_2|^{\frac{d-2}{q}}}\frac{1}{|x_3-x_4|^{\frac{d-2}{q}}}dx \les  \eps^{\frac{2(d-2)}{q}-\frac{4}{p}}.
\end{aligned}
\]
We applied Lemma~\ref{l.qmm} in the next to last step. It remains 
to choose $p$ large enough so
that 
\[
\frac{2(d-2)}{q}-\frac{4}{p}=\frac{2d}{q}-4>2(d-2)-4\delta
\]
to complete the proof.
\end{proof}

\subsection{The second derivative}

To estimate $\|\cP_k\|_{\mathrm{op}}$, we use the contraction inequality \cite[Proposition 4.1]{nourdin2009second}, which says that 
\[
\|\cP_k\|_{\mathrm{op}}^4\leq \|\cP_k\otimes_1\cP_k\|_{H\otimes H}^2.
\]
Here $\cP_k\otimes_1\cP_k$ is the random element of $H\otimes H$ obtained as the contraction of the symmetric
random tensor $\cP_k$.

\subsubsection{The case $k=1$} 

A direct calculation gives 
\[
\begin{aligned}
&\cP_1\otimes_1\cP_1
=\beta^4\int_{\R^{2d}} \frac{\E_B[ \prod_{j=1}^4 M_{\eps,j}(t,x_j) \cR(\frac{t}{\eps^2},\frac{x-y}{\eps},B^1,B^3)\Phi^\eps_{t,x,B^2}\otimes \Phi^\eps_{t,y,B^4}]}{Z_\eps^2(t,x)Z_\eps^2(t,y)} g(x)g(y)dxdy,
\end{aligned}
\]
where we write $x_1=x_2=x,x_3=x_4=y$ to simplify the notations. Thus, 
\[
\begin{aligned}
\|\cP_1\otimes_1\cP_1\|_{H\otimes H}^2=\beta^8\int_{\R^{4d}}&g(x)g(y)g(z)g(w)
\Big(\prod_{j=1}^8 Z_\eps(t,x_j)\Big)^{-1} \\
&\times \E_B\Big[ \prod_{j=1}^8 M_{\eps,j}(t,x_j)\prod_{(i,k)\in \mathcal{O}} \cR(\frac{t}{\eps^2},\frac{x_i-x_k}{\eps},B^i,B^k) \Big]dxdydzdw,
\end{aligned}
\]
where $x_5=x_6=z, x_7=x_8=w$, and the set $\mathcal{O}$ is 
\[
\mathcal{O}=\{(1,3),(5,7),(2,6),(4,8)\}.
\]
\begin{lemma}\label{l.p1}
For any $\delta>0$, there exists $\beta(\delta)$ such that if $\beta<\beta(\delta)$,  
\[
\E[\|\cP_1\otimes_1\cP_1\|_{H\otimes H}^2]\les \eps^{4d-8-\delta}\1_{d<8} +\eps^{3d-8-\delta}\1_{d\geq 8}.
\]
\end{lemma}

\begin{proof}
Applying Lemma~\ref{l.holder}, we have 
\[
\begin{aligned}
\E[\|\cP_1\otimes_1\cP_1\|_{H\otimes H}^2] \les& \int_{\R^{4d}} |g(x)g(y)g(z)g(w)| \times\E_B\Big[\prod_{(i,k)\in \mathcal{O}} 
\cR^q(\frac{t}{\eps^2},\frac{x_i-x_k}{\eps},B^i,B^k)\Big]^{1/q} dxdydzdw\\
=&\int_{\R^{4d}} |g(x)g(y)g(z)g(w)|\times \prod_{(i,k)\in \mathcal{O}} 
\E_B\Big[\cR^q(\frac{t}{\eps^2},\frac{x_i-x_k}{\eps},B^i,B^k)\Big]^{1/q}dxdydzdw
\end{aligned}
\]
for some $q>1$. We discuss two cases.

(i) $d<8$. Applying Lemma~\ref{l.qmm} to all pairs $(i,k)\in \mathcal{O}$, the above integral is bounded by 
\[
\eps^{4(\frac{d-2}{q}-\frac{2}{p})}\int_{\R^{4d}} \frac{ |g(x)g(y)g(z)g(w)|}{(|x-y||z-w||x-z||y-w|)^{(d-2)/q}}dxdydzdw.
\]
Since $g\in\C_c$, by the elementary inequality 
\[
\int_{|y|\leq M} \frac{1}{|x-y|^{\alpha_1}}\frac{1}{|y|^{\alpha_2}}dy\les \frac{\1_{\alpha_1+\alpha_2>d}}{|x|^{\alpha_1+\alpha_2-d}}+\1_{\alpha_1+\alpha_2<d}, \quad       \mbox{ if }\quad \alpha_1<d,\alpha_2<d, \alpha_1+\alpha_2\neq d,
\]
 the above integral is bounded in $d<8$. Thus, we have, for $p$ sufficiently large,
\[
\E[\|\cP_1\otimes_1\cP_1\|_{H\otimes H}^2] \les \eps^{4d-8-\delta}.
\]

(ii) $d\geq 8$. Applying
Lemma~\ref{l.qmm} to three pairs of $(i,k)\in \mathcal{O}$, and bounding
the fourth pair simply by $\eps^{-2}$, and using the above elementary inequality again, we have for large $p$ that 
\[
\begin{aligned}
\E[\|\cP_1\otimes_1\cP_1\|_{H\otimes H}^2]& \les \eps^{3(\frac{d-2}{q}-\frac{2}{p})-2} \int_{\R^{4d}}\frac{ |g(x)g(y)g(z)g(w)|}{(|x-y||z-w||x-z|)^{(d-2)/q}}dxdydzdw\\
&\les  \eps^{3(d-2)-2-\delta}.
\end{aligned}
\]
The proof is complete.
\end{proof}

\subsubsection{The case $k=2$} 

In this case, 
\[
\cP_2= \beta^2\int_{\R^d} \frac{\E_B[ M_\eps(t,x)\Phi^\eps_{t,x,B}\otimes \Phi^\eps_{t,x,B}]}{Z_\eps(t,x)}g(x)dx,
\]
so 
\[
\begin{aligned}
&\cP_2\otimes_1\cP_2= \beta^4\int_{\R^{2d}} \frac{\E_B[ \prod_{j=1}^2 M_{\eps,j}(t,x_j) \cR(\frac{t}{\eps^2},\frac{x_1-x_2}{\eps},B^1,B^2)\Phi^\eps_{t,x_1,B^1}\otimes \Phi^\eps_{t,x_2,B^2}]}{Z_\eps(t,x_1)Z_\eps(t,x_2)} g(x_1)g(x_2)dx_1dx_2,
\end{aligned}
\]
and
\[
\|\cP_2\otimes_1 \cP_2\|_{H\otimes H}^2=\beta^8 \int_{\R^{4d}} \prod_{j=1}^4 \frac{g(x_j)}{Z_\eps(t,x_j)} \E_B\Big[ \prod_{j=1}^4 M_{\eps,j}(t,x_j) \prod_{(i,k)\in \tilde{\mathcal{O}}} \cR(\frac{t}{\eps^2},\frac{x_i-x_k}{\eps},B^i,B^k)\Big]dx,
\]
with 
\[
\tilde{\mathcal{O}}=\{(1,2),(3,4),(1,3),(2,4)\}.
\]

\begin{lemma}\label{l.p2}
For any $\delta>0$, there exists $\beta(\delta)$ such that if $\beta<\beta(\delta)$, 
\[
\E[\|\cP_2\otimes_1 \cP_2\|_{H\otimes H}^2]\les \eps^{3(d-2)-\delta}.
\]
\end{lemma}

\begin{proof}
By Lemma~\ref{l.holder} and the fact that $g$ is compactly supported, we have 
\[
\begin{aligned}
\E[\|\cP_2\otimes_1 \cP_2\|_{H\otimes H}^2] \les &\int_{\R^{4d}} \prod_{j=1}^4 |g(x_j)|\times \E_B\left[ \prod_{(i,k)\in \tilde{\mathcal{O}}} \cR^q(\frac{t}{\eps^2},\frac{x_i-x_k}{\eps},B^i,B^k)\right]^{1/q}dx\\
\les &\Big(\int_{\R^{4d}} \prod_{j=1}^4|g(x_j)|^q\times \E_B\Big[   
\prod_{(i,k)\in \tilde{\mathcal{O}}} 
\cR^q(\frac{t}{\eps^2},\frac{x_i-x_k}{\eps},B^i,B^k)\Big] dx\Big)^{1/q}.
\end{aligned}
\]
Since 
\[
\begin{aligned}
\cR^q(\frac{t}{\eps^2},\frac{x_i-x_k}{\eps},B^i,B^k)=
\Big(\int_0^{t/\eps^2}\!\!
R(\frac{x_i-x_k}{\eps}+B^i_s-B^k_s)ds\Big)^q
\leq \Big(\frac{t}{\eps^2}\Big)^{q/p}\int_0^{t/\eps^2}\!\! 
R^q(\frac{x_i-x_k}{\eps}+B^i_s-B^k_s)ds,
\end{aligned}
\]
we only need to control 
\[
\int_{\R^{4d}}  \prod_{j=1}^4|g(x_j)|^q \times
\E_B\Big[  \prod_{(i,k)\in \tilde{\mathcal{O}}} \int_0^{t/\eps^2} R^q(\frac{x_i-x_k}{\eps}+B^i_s-B^k_s)ds\Big]dx.
\]
Applying Lemma~\ref{l.bin}, we have 
\[
\E[\|\cP_2\otimes \cP_2\|_{H\otimes H}^2] \les 
\Big(\frac{1}{\eps^2}\Big)^{4/p} \eps^{3(d-2)/q}=\eps^{\frac{3(d-2)}{q}-\frac{8}{p}}.
\]
The proof is complete.
\end{proof}

\begin{lemma}\label{l.bin}
Assume $0\leq f,h\in\C_c^\infty(\R^d)$, then
\[
 \int_{\R^{4d}}  \E_B\Big[ \prod_{j=1}^4f(x_j) 
 \prod_{(i,k)\in \tilde{\mathcal{O}}} \int_0^{t/\eps^2} 
 h(\frac{x_i-x_k}{\eps}+B^i_s-B^k_s)ds\Big]dx \les \eps^{3(d-2)}.
 \]
\end{lemma}

\begin{proof}
Without loss of generality, assume $h$ is even. To simplify the notation, we write 
\[
\begin{aligned}
\prod_{(i,k)\in \tilde{\mathcal{O}}} \int_0^{t/\eps^2} h(\frac{x_i-x_k}{\eps}+B^i_s-B^k_s)ds=\int_{[0,t/\eps^2]^4}\prod_{j=1}^4 h(\frac{x_j-x_{j-1}}{\eps}+B^j_{s_j}-B^{j-1}_{s_j})ds\\
=\sum_{\tau}\int_{A_\tau}\prod_{j=1}^4 h(\frac{x_j-x_{j-1}}{\eps}+B^j_{s_j}-B^{j-1}_{s_j})ds,
\end{aligned}
\]
where we let $x_0=x_4,B^0=B^4$, $\tau$ denotes the permutations of $s_1,\ldots,s_4$, and $A_\tau\subset [0,t/\eps^2]^4$ corresponds to the permutation $\tau$. Due to symmetry, there are six different permutations to consider. 

Now we write the integral in the Fourier domain. Denote $\hat{f}(\xi)=\int f(x)e^{-i\xi\cdot x}dx$ as the Fourier transform of $f$, we have
\[
\begin{aligned}
\int_{\R^{4d}} \prod_{j=1}^4 f(x_j)h(\frac{x_j-x_{j-1}}{\eps}+B^j_{s_j}-B^{j-1}_{s_j}) dx&=\frac{1}{(2\pi)^{4d}}\int_{\R^{8d}} \prod_{j=1}^4 f(x_j)\hat{h}(\eta_j) e^{i\eta_j\cdot (x_j-x_{j-1})/\eps}e^{i\eta_j \cdot(B^j_{s_j}-B^{j-1}_{s_j})} d\eta dx\\
&=\frac{1}{(2\pi)^{4d}} \int_{\R^{4d}} \prod_{j=1}^4 \hat{f}(\frac{\eta_{j+1}-\eta_{j}}{\eps}) \hat{h}(\eta_j)  e^{i(\eta_j\cdot B^j_{s_j}-\eta_{j+1}\cdot B^j_{s_{j+1}})}d\eta,
\end{aligned}
\]
with $\eta_5=\eta_1, s_5=s_1$. Thus, it suffices to estimate 
\[
\int_{A_\tau}\int_{\R^{4d}}\prod_{j=1}^4 \hat{f}(\frac{\eta_{j+1}-\eta_{j}}{\eps}) \hat{h}(\eta_j)  \E_B[e^{i(\eta_j\cdot B_{s_j}-\eta_{j+1}\cdot B_{s_{j+1}})}]d\eta.
\]
First, we change variables 
\[
\eta_1= \xi_1,\    \  \eta_j=\xi_1+\eps(\xi_2+\ldots+\xi_j),  \  \ j=2,3,4,
\]
and the above integral equals to 
\[
\eps^{3d}\int_{A_\tau} \int_{\R^{4d}} \hat{h}(\xi_1)\hat{f}(\xi_2)\hat{f}(\xi_3)\hat{f}(\xi_4) w_\eps(\xi) \prod_{j=1}^4  \E_B[e^{i(\eta_j\cdot B_{s_j}-\eta_{j+1}\cdot B_{s_{j+1}})}] d\xi
\]
with 
\[
w_\eps(\xi)=\hat{f}(-\xi_2-\xi_3-\xi_4)\hat{h}(\eta_2)\hat{h}(\eta_3)\hat{h}(\eta_4)\in L^\infty(\R^{4d}).
\]
Depending on the permutation $\tau$, the factor $\prod_{j=1}^4  \E_B[e^{i(\eta_j\cdot B_{s_j}-\eta_{j+1}\cdot B_{s_{j+1}})}]$ can be computed explicitly. Since all six cases are treated in the same way, we only take $s_1<s_2<s_3<s_4$ as an example:
\[
\begin{aligned}
\prod_{j=1}^4  \E_B[e^{i(\eta_j\cdot B_{s_j}-\eta_{j+1}\cdot B_{s_{j+1}})}]&=e^{-\frac12|\eta_4-\eta_1|^2s_1}e^{-\frac12|\eta_4|^2(s_4-s_1)}\prod_{j=1}^3 e^{-\frac12|\eta_j-\eta_{j+1}|^2s_j} e^{-\frac12|\eta_{j+1}|^2 (s_{j+1}-s_j)}\\
&=e^{-\frac12\sum_{j=1}^4 \lambda_j(s_j-s_{j-1})}
\end{aligned}
\]
with $s_0=0$ and 
\[
\begin{aligned}
&\lambda_1=\eps^2(|\xi_2|^2+|\xi_3|^2+|\xi_4|^2+|\xi_2+\xi_3+\xi_4|^2),\\
 &\lambda_2=\eps^2|\xi_3|^2+\eps^2|\xi_4|^2+|\xi_1+\eps\xi_2|^2+|\xi_1+\eps(\xi_2+\xi_3+\xi_4)|^2,\\
&\lambda_3=\eps^2|\xi_4|^2+|\xi_1+\eps \xi_2+\eps \xi_3|^2+|\xi_1+\eps(\xi_2+\xi_3+\xi_4)|^2,\\
 &\lambda_4=2|\xi_1+\eps(\xi_2+\xi_3+\xi_4)|^2.
\end{aligned}
\]
After integrating the $s$ variables, we have 
\[
\int_{0<s_1<\ldots<s_4<t/\eps^2}\prod_{j=1}^4  \E_B[e^{i(\eta_j\cdot B_{s_j}-\eta_{j+1}\cdot B_{s_{j+1}})}]ds \les  \frac{1}{\lambda_1\lambda_2\lambda_3\lambda_4} \les  \frac{\eps^{-6}}{|\xi_2|^2|\xi_3|^2|\xi_4|^2|\xi_1+\eps(\xi_2+\xi_3+\xi_4)|^2}
\]
In the end, we note that 
\[
\int_{\R^{4d}}\frac{|\hat{h}(\xi_1)\hat{f}(\xi_2)\hat{f}(\xi_3)\hat{f}(\xi_4)|}{|\xi_2|^2|\xi_3|^2|\xi_4|^2|\xi_1+\eps(\xi_2+\xi_3+\xi_4)|^2}d\xi\les1
\]
to complete the proof.
\end{proof}

\subsection{Proof of Proposition~\ref{p.gauss}}
Recall that 
\[
Y_\eps=\frac{X_\eps-\E[X_\eps]}{\sqrt{\Var[X_\eps]}}.
\]
Since 
\[
\begin{aligned}
d_{\mathrm{TV}}(Y_\eps,\zeta) &\les \E[\|DY_\eps\|_H^4]^{1/4} \E[\|D^2Y_\eps\|_{\mathrm{op}}^4]^{1/4}\\
&=\frac{1}{\Var[X_\eps]}\E[\|DX_\eps\|_H^4]^{1/4} \E[\|D^2X_\eps\|_{\mathrm{op}}^4]^{1/4},
\end{aligned}
\]
using the fact that $\Var[X_\eps]\sim \eps^{d-2}$ and  applying Lemmas~\ref{l.1stde}, \ref{l.p1} and \ref{l.p2}, we have 
\[
d_{\mathrm{TV}}(Y_\eps,\zeta) \les \eps^{2-d} \eps^{\frac{d-2}{2}-\delta}\left(\eps^{4d-8-\delta}\1_{d<8} +\eps^{3d-8-\delta}\1_{d\geq 8}+\eps^{3(d-2)-\delta}\right)^{\frac14}.
\]
By choosing $\delta$ small, the r.h.s. goes to zero as $\eps\to0$.

\section{Proof of Theorem~\ref{t.generalf}}
The proof of Theorem~\ref{t.mainth} applies almost verbatim with the logarithm $\log y$ replaced by a
general function $\mathfrak{f}(y)$. 
We need to use the assumption 
\[
|\mathfrak{f}(y)|+|\mathfrak{f}'(y)|+|\mathfrak{f}''(y)| \leq M(y^p+y^{-p})
\]
to guarantee that 
\[
\E[|\mathfrak{g}(u_\eps(t,x))|^n] \les1
\]
for $\mathfrak{g}\in\{\mathfrak{f},\mathfrak{f}',\mathfrak{f}''\}$,
provided that $\beta$ is chosen small. The only changes needed are in Section~\ref{s.analysisi3}, and we sketch them here. 

First, we have for general $\mathfrak{f}$ that 
\[
I_{3,\eps}=\int_{K}^{t/\eps^2} \int_{\R^d}\Big(\int_{\R^d} g(x)\mathfrak{f}'(Z(K,x/\eps))
\E\Big[ \E_B[ M_\eps(t,x)\Phi^\eps_{t,x,B}(s,y)]\big| \F_s\Big]dx \Big) dW(s,y).
\]
The following two lemmas come from the same proof of Lemmas~\ref{l.i3re} and \ref{l.i3bd}.
\begin{lemma}\label{l.i3renew}
\[
\eps^{-(d-2)}\E[I_{3,\eps}^2]= \int_{0}^{t-\eps^2K} \G_\eps(s) ds,
\]
with 
\begin{equation}
\begin{aligned}
\G_\eps(s)=\int_{\R^{3d}}&g(x-w)g(x) R(y)  \E\Big[\mathfrak{f}'(Z(K,-\frac{w}{\eps}))\mathfrak{f}'(Z(K,0))\\
&\times \E_B[M_1(K,-\frac{w}{\eps})M_2(K,0)G_{2s}(w+\eps y-\eps \X_K)\cH(\frac{s}{\eps^2},y,\X_K-\frac{w}{\eps}-y)]\Big]dxdydw.
\end{aligned}
\end{equation}
\end{lemma}

\begin{lemma}\label{l.i3bdnew}
There exists $\beta_0>0$ so that there exists $\gamma\in (0,1)$ such that,  for all  $\beta<\beta_0$,
$\G_\eps(s)\les s^{-\gamma}$ for $s\in(0,t)$.
\end{lemma}

It remains to show the following lemma to complete the proof of Theorem~\ref{t.generalf}.
\begin{lemma}\label{l.i3limnew}
For any $s\in(0,t)$, 
\[
\G_\eps(s)\to \nu_{\mathrm{eff}}^2 \sigma_{\mathfrak{f}}^2\int_{\R^{2d}} g(x-w)g(x)G_{2s}(w)dwdx,
\]
 as $\eps\to0$.
\end{lemma}

\begin{proof}
As in the proof of Lemma~\ref{l.i3lim}, we have
\[
\G_\eps(s)-\nu_{\mathrm{eff}}^2\int_{\R^{2d}}g(x-w)g(x)G_{2s}(w)\E[\mathfrak{f}'(Z(K,-\frac{w}{\eps}))\mathfrak{f}'(Z(K,0))\E_B[M_1(K,-\frac{w}{\eps})M_2(K,0)]]\to0,
\]
so it remains to analyze 
\[
\begin{aligned}
&\E[\mathfrak{f}'(Z(K,-\frac{w}{\eps}))\mathfrak{f}'(Z(K,0))\E_B[M_1(K,-\frac{w}{\eps})M_2(K,0)]]\\
&=\E[\mathfrak{f}'(Z(K,-\frac{w}{\eps}))\mathfrak{f}'(Z(K,0))Z(K,-\frac{w}{\eps})Z(K,0)]=\E[\zeta(K,-\frac{w}{\eps})\zeta(K,0)],
\end{aligned}
\]
where we defined \[
\zeta(t,x):=\mathfrak{f}'(Z(t,x))Z(t,x).
\] By stationarity in the $x-$variable, we write 
\[
\E[\zeta(K,-\frac{w}{\eps})\zeta(K,0)]=\mathrm{Cov}[\zeta(K,-\frac{w}{\eps}),\zeta(K,0)]+\E[\zeta(K,0)]^2.
\]
By \eqref{e.zinfinity}, we have $Z(t,0)\to Z_\infty$ almost surely, thus, as $K\to\infty$, we have
\[
\E[\zeta(K,0)]\to \sigma_{\mathfrak{f}}=\E[\mathfrak{f}'(Z_\infty)Z_\infty].
\] 
The following lemma completes the proof.
\begin{lemma}\label{l.covdecay}
For any $\delta>0$, if $\beta<\beta(\delta)$, we have for all $x\neq0$ that 
\[
\sup_{t\in[0,\eps^{-2}]}\mathrm{Cov}[\zeta(t,x),\zeta(t,0)]\les  \frac{\eps^{-\delta}}{|x|^{d-2}}.
\]
\end{lemma}

\begin{proof}
By the covariance inequality, we have 
\[
|\mathrm{Cov}[\zeta(t,x),\zeta(t,0)]|\leq  \int_{\R^{d+1}} \sqrt{\E[|D_{r,y}\zeta(t,x)|^2]\E[|D_{r,y}\zeta(t,0)|^2]} dydr.
  \]
Since we have
\begin{equation}\label{dec906}
D_{r,y}\zeta=[\mathfrak{f}''(Z)Z+\mathfrak{f}'(Z)] D_{r,y}Z,
\end{equation}
the Cauchy-Schwarz inequality yields
  \[
  |\mathrm{Cov}[\zeta(t,x),\zeta(t,0)]|\leq \int_{\R^{d+1}} \E[|D_{r,y}Z(t,x)|^4]^{1/4}\E[|D_{r,y}Z(t,0)|^4]^{1/4}dydr.
  \]
Recalling that 
\[
D_{r,y}Z(t,x)=\beta\E_B[ M(t,x)  \1_{[0,t]}(r)\varphi(x+B_r-y)],
\]
by Lemmas~\ref{l.holder} and \ref{l.exmm} we obtain
\[
\E[|D_{r,y}Z(t,x)|^4]^{1/4} \les \1_{[0,t]}(r)\E_B[\varphi^q(x+B_r-y)]^{1/q},
\]
for any $q>1$, which implies 
\begin{equation}\label{e.1251}
|\mathrm{Cov}[\zeta(t,x),\zeta(t,0)]|\les\int_0^t \int_{\R^d}\E_B[\varphi^q(x+B_r-y)]^{1/q}\E_B[\varphi^q(B_r-y)]^{1/q}dydr.
\end{equation}
As $\varphi\in \C_c(\R^d)$, we use the simple bound 
\[
\E_B[ \varphi^q(z+B_r)]^{1/q} \les \Pb[|z+B_r|\leq 1]^{1/q}\les \left(\1_{|z|\leq C}+r^{-d/2}e^{-\frac{|z|^2}{Cr}}\1_{|z|>C}\right)^{1/q}\les \1_{|z|\leq C}+t^{\frac{d}{2p}}r^{-d/2}e^{-\frac{|z|^2}{qCr}}\1_{|z|>C},
\]
for all $z\in\R^d$ and $r\in(0,t)$, where $C$ is some positive constant. The r.h.s. of \eqref{e.1251} is then bounded by
\[
\int_0^t\int_{\R^d}\left(\1_{|x-y|\leq C}+t^{\frac{d}{2p}}r^{-d/2}e^{-\frac{|x-y|^2}{qCr}}\1_{|x-y|>C}\right)\left(\1_{|y|\leq C}+t^{\frac{d}{2p}}r^{-d/2}e^{-\frac{|y|^2}{qCr}}\1_{|y|>C}\right)dydr.
\]
For  $|x|\gg1$, integrating in $r$ in the above expression to derive
\[
|\mathrm{Cov}[\zeta(t,x),\zeta(t,0)]|\les  \frac{t^{\frac{d}{p}}+t^{\frac{d}{2p}}}{|x|^{d-2}}.
\]
It suffices to pick $p\gg1$ to complete the proof.
\end{proof}

This also finishes the proof of Theorem~\ref{t.generalf}.
\end{proof}

\begin{remark}\label{r.sigmaf}
By \cite[Theorem 1.2]{dunlap2018random}, the effective variance in \eqref{e.defU} is related to the asymptotic decorrelation rate of the stationary 
solution $\Psi(t,x)$ of the stochastic heat equation,  
as in (\ref{dec902}):
\begin{equation}\label{dec916bis}
\mathrm{Cov}[\Psi(t,x),\Psi(t,y)] \approx 
\frac{c(\beta)\nu_{\mathrm{eff}}^2}{|x-y|^{d-2}},\hbox{ for $|x-y|\gg 1$,} 
\end{equation}
with $c(\beta)=\bar c\beta^2$. 
Theorem~\ref{t.generalf} further indicates that  for any smooth function $\mathfrak f$
we also have
\begin{equation}\label{e.1283bis}
\begin{aligned}
\mathrm{Cov}[\mathfrak{f}(\Psi(t,x)),\mathfrak{f}(\Psi(t,y))]
\approx
\frac{\sigma_{\mathfrak{f}}^2c(\beta)\nu_{\mathrm{eff}}^2}{|x-y|^{d-2}}, \hbox{ for $|x-y|\gg 1$,}
\end{aligned}
\end{equation}
with $\sigma_{\mathfrak{f}}=\E[\mathfrak{f}'(\Psi(t,x))\Psi(t,x)]$,
so that
\begin{equation}\label{dec1038}
\begin{aligned}
\mathrm{Cov}[\mathfrak{f}(\Psi(t,x)),\mathfrak{f}(\Psi(t,y))]
\approx\sigma_{\mathfrak{f}}^2 \mathrm{Cov}[\Psi(t,x),\Psi(t,y)] 
\hbox{ for $|x-y|\gg 1$.}
\end{aligned}
\end{equation}
Recall also that 
\begin{equation}\label{dec1040}
\E[\Psi(t,x)]= \E[Z_\infty]=1.
\end{equation}
We 
restate these properties in terms of the stationary solution $H(t,x)=\log\Psi(t,x)$
of the KPZ equation:
\begin{equation}\label{dec1026}
\mathrm{Cov}[e^{H(t,x)},e^{H(t,y)}] \approx 
\frac{c(\beta)\nu_{\mathrm{eff}}^2}{|x-y|^{d-2}},\hbox{ for $|x-y|\gg 1$,} 
\end{equation}
and
\begin{equation}\label{dec1028}
\begin{aligned}
\mathrm{Cov}[\mathfrak{f}(e^{H(t,x)}),\mathfrak{f}(e^{H(t,y)})]
\approx
\frac{\sigma_{\mathfrak{f}}^2c(\beta)\nu_{\mathrm{eff}}^2}{|x-y|^{d-2}}
\approx \sigma_{\mathfrak{f}}^2 \mathrm{Cov}[e^{H(t,x)},e^{H(t,y)}], 
\hbox{ for $|x-y|\gg 1$,}
\end{aligned}
\end{equation}
with $\sigma_{\mathfrak{f}}=\E[\mathfrak{f}'(e^{H(t,x)})e^{H(t,x)}]$, and 
\begin{equation}\label{dec1030}
\E[e^{H(t,x)}]=  1.
\end{equation}
We can now illustrate the origin of $\sigma_{\mathfrak{f}}$ through a 
toy calculation. 
Let  $X$ and $Y$ be two jointly 
Gaussian~$N(0,1)$ variables,  with $\mathrm{Cov}[X,Y]=\delta\ll 1$, and
define $\mathcal{X}=e^{X-\frac12}, \mathcal{Y}=e^{Y-\frac12}$, so that 
\begin{equation}\label{dec920}
\E[\mathcal{X}]=\E[\mathcal{Y}]=1.
\end{equation}
We think of $\mathcal{X}$ as representing $e^{H(t,x)}$ and $\mathcal{Y}$
as representing $e^{H(t,y)}$, so as to fit \eqref{dec1030},
although we emphasize that there is no real 
claim of Gaussianity of $H(t,\cdot)$ in the pointwise sense. 
With this approximation, we may write
\[
Y=\delta X+\sqrt{1-\delta^2} W,
\]
with $W$ another~$N(0,1)$ variable independent of $X$. Then, first,
we have
\[
\mathrm{Cov}[\mathcal{X},\mathcal{Y}]=\E[\mathcal{X}\mathcal{Y}]-\E[\mathcal{X}]\E[\mathcal{Y}]=
e^{\delta}-1= \delta+o(\delta),
\]
and second,
we have 
\begin{equation}\label{dec912}
\begin{aligned}
\mathrm{Cov}[\mathfrak{f}(\mathcal{X}),\mathfrak{f}(\mathcal{Y})]=\E[\mathfrak{f}(\mathcal{X})\mathfrak{f}(\mathcal{Y})]-\E[\mathfrak{f}(\mathcal{X})]\E[\mathfrak{f}(\mathcal{Y})]
&= \delta \E[\mathfrak{f}(e^{X-\frac12})X]\E[\mathfrak{f}'(e^{W-\frac12})e^{W-\frac12}]+o(\delta)\\
&=\delta \sigma_{\mathfrak{f}}^2+o(\delta).
\end{aligned}
\end{equation}
We have denoted $\sigma_{\mathfrak{f}}=\E[ \mathfrak{f}'(\mathcal{X})\mathcal{X}]$ for smooth functions $\mathfrak{f}$,
as before. Here, we used the identity 
\begin{equation}\label{dec914}
\E[\mathfrak{f}(e^{X-\frac12})X]=\E[\mathfrak{f}'(e^{X-\frac12})e^{X-\frac12}],
\end{equation}
obtained via integration by parts. 
Thus, we have
\begin{equation}\label{e.1282}
 {\mathrm{Cov}[\mathfrak{f}(\mathcal{X}),\mathfrak{f}(\mathcal{Y})]} \approx  \sigma_{\mathfrak f}^2 \mathrm{Cov}[\mathcal{X},\mathcal{Y}].
\end{equation}
Unravelling the definitions, this parallels
\eqref{dec1028}.

As emphasized above, there was nothing Gaussian in the pointwise sense in
the field $H(t,x)$. However, the above computation 
would be essentially unchanged if we replace $H(t,x)$ by the field
%
%
\begin{equation}\label{dec1048}
\tilde H(t,x)={\mathcal G}(t,x)+E(t,x),
\end{equation}
where ${\mathcal G}(t,\cdot)$ is a mean-zero spatially  stationary
Gaussian random field with $\E[\calG(t,x)^2]=\sigma^2$, 
and correlation function
\begin{equation}\label{dec1042}
\E[{\mathcal G}(t,x){\mathcal G}(t,y)]\approx\farc{\sigma^2}{|x-y|^{d-2}},\hbox{
for $|x-y|\gg 1$},
\end{equation}
while $E(t,\cdot)$ is a spatial stationary random field of negative mean,
independent of ${\mathcal G}(t,\cdot)$,  
with a rapidly (in space) decaying correlation
function, 
such that
\begin{equation}\label{dec1032}
\E[e^{\calG(t,x)}]\E[e^{E(t,x)}]=  1.
\end{equation}
(This would be consistent with (\ref{dec1026})-(\ref{dec1028}).)
Repeating the above computation with 
$x,y\in\R^d$ such that $|x-y|\gg 1$ and 
$\delta=\mathrm{Cov}[\calG(t,x),\calG(t,y)]\ll 1$, we then obtain
again that
\begin{equation}\label{dec1046}
\mathrm{Cov}[\mathfrak{f}(e^{\tilde H(t,x)}),\mathfrak{f}(e^{\tilde H(t,y)})]  
\approx  \sigma_{\mathfrak f}^2 \mathrm{Cov}[e^{\tilde H(t,x)},e^{\tilde H(t,y)} ].
\end{equation}
Thus, any field of the form (\ref{dec1048})
satisfies
conditions~(\ref{dec1026})-(\ref{dec1030}), 
with the variance $\sigma_{\mathfrak f}$ as in
Theorem~\ref{t.generalf}. 
It is tempting to speculate that 
the stationary solution $H(t,x)$ of the KPZ equation has a decomposition 
of the form
\eqref{dec1048}; we unfortunately do not have any supporting evidence for such
a speculation,
and at this point cannot even agree whether such a decomposition is behind  the results of this paper or, alternatively, that temporal mixing
plays an additional important role.
%
\end{remark}

\appendix

\section{Auxiliary lemmas}

\begin{lemma}\label{l.delog}
For any $t>0,x\in\R^d$, we have
\[
D\log Z(t,x)=\frac{DZ(t,x)}{Z(t,x)}\in L^2(\Omega;H).
\]
\end{lemma}

\begin{proof}
Recall that 
\[
\begin{aligned}
&Z(t,x)=\E_B\left[ \exp\Big\{\beta\int_0^t V(s,x+B_s)ds-\frac12\beta^2 R(0)t\Big\}\right],  \\
&\int_0^t V(s,x+B_s)ds=\int_{\R^{d+1}} \1_{[0,t]}(s) \varphi(x+B_s-y) dW(s,y),
\end{aligned}
\]
so, for each $t$ and $x$ fixed,  we have
\begin{equation}\label{nov2710}
D_{s,y}Z(t,x)=\beta\E_B\left[ \exp\Big\{
\beta\int_0^t V(s,x+B_s)ds-\frac12\beta^2 R(0)t\Big\}
\1_{[0,t]}(s) \varphi(x+B_s-y) \right] \in L^n(\Omega; H),
\end{equation}
for any $n\in\Z_+$, where we recall that
$H$ is the $L^2(\R^{d+1})$-space with respect to the $s,y$-variables.
To deal with the logarithm function,
which is singular at the origin and grows at infinity, we use an approximation~$f_n\in \C_c^\infty(\R)$ such that $f_n(x)=\log x$ for $x\in[1/n,n]$ and $|f_n'(x)|\leq |x|^{-1}$. It is clear that 
\[
D f_n(Z(t,x))=f_n'(Z(t,x)) DZ(t,x)\in L^2(\Omega; H),
\]
and the error 
\[
\begin{aligned}
\E\left[\bigg\| Df_n(Z(t,x))-\frac{DZ(t,x)}{Z(t,x)}\bigg\|_H^2\right]\les& \E\left[ \frac{ \|DZ(t,x)\|_H^2}{|Z(t,x)|^2}\big(\1_{Z(t,x)<\frac{1}{n}}+\1_{Z(t,x)>n}\big)\right]\\
\les & \sqrt{\E[|Z(t,x)|^{-4} \big(\1_{Z(t,x)<\frac{1}{n}}+\1_{Z(t,x)>n}\big)]}\to0
    \end{aligned}
    \]
   as $n\to\infty$, where we used Proposition~\ref{p.negative} in the last step,
together with (\ref{nov2710}). By \cite[Proposition 1.2.1]{nualart2006malliavin}, the proof is complete.
\end{proof}

\begin{lemma}\label{l.exmm}
There exists $\beta_0>0$ such that if $\beta<\beta_0$, we have in $d\geq3$ that  
\begin{equation}
\label{eq:BMconcl}
\sup_{x\in\R^d} \E_B\left[ \exp\Big\{\beta\int_0^\infty R(x+B_s)ds\Big\}\right]<\infty,
\end{equation}
and 
\begin{equation}
\label{eq:bridgeconcl}
\sup_{t>0, x,y\in\R^d} \E_B\left[ \exp\Big\{
\beta\int_0^t R(x+B_s)ds\Big\} \bigg| B_t=y\right]<\infty.
\end{equation}
\end{lemma}
\begin{proof}
  The statement in \eqref{eq:BMconcl} follows from Portenko's lemma, see 
  \cite[(3.1), (3.2)]{mukherjee2016weak}.

  We turn to proving  \eqref{eq:bridgeconcl}.
  Conditioned on $B_t=y$, the process
$\{B_s\}_{s\leq t}$ is a Brownian bridge. 
In particular, it has a Markovian representation. Thus, 
  again by Portenko's lemma, it is enough to show that 
 \begin{equation}
\label{eq:bridgeconcl1a}
\beta 
\sup_{t>0, x,y\in\R^d} \E_B\left[ \int_0^{t} R(x+B_s)ds \bigg| B_t=y\right]<
1,
\end{equation} 
for all $\beta$ small enough. By symmetry, as $X_s$ is a Brownian bridge,
it suffices to show that
 \begin{equation}
\label{eq:bridgeconcl2}
\sup_{t>0, x,y\in\R^d} \E_B\left[ \int_0^{t/2} R(x+B_s)ds \bigg| B_t=y\right]
<
\infty.
\end{equation}
Note that $X_s$ has mean $sy/t$ and variance $s(t-s)/t$, which in our
range is larger than $Cs$. In particular, for $s\leq t/2$, we have
$$\sup_{x,y\in\R^d}\Pb_B\left[|B_s-x|\leq 1\bigg| B_t=y\right]\leq C{(1+}s)^{-d/2}.$$
Since $s^{-d/2}$ is integrable as $s\to+\infty$, 
this yields \eqref{eq:bridgeconcl2}
and completes the proof of the lemma.
\end{proof}

\begin{lemma}\label{l.bridgetail}
Let $B$ be a standard Brownian motion in $d\geq3$. For any $\alpha\in(0,2)$, $t>0$ and compact set $A\subset \R^d$ with $0\nin A$, we have
\[
  \sup_{w\in A}\E_{B}\left[ \big(\int_{t/\eps^\alpha}^{t/\eps^2}R(B_s)ds\big)^2 \,\bigg|\,   B_{t/\eps^2}=\frac{w}{\eps}\right]\les \eps^{\alpha(\frac{d}{2}-1)}.
\]
\end{lemma}
\begin{proof}
By a direct calculation, we have 
\[
\begin{aligned}
&\E_{B}\left[ \big(\int_{t/\eps^\alpha}^{t/\eps^2}R(B_s)ds\big)^2 \,\bigg|\,   B_{t/\eps^2}=\frac{w}{\eps}\right]\\
&=2\int_{[t/\eps^\alpha\leq s_1 \leq s_2 \leq t/\eps^2]}\int_{\R^{2d}}R(x)R(y)G_{s_1}(x)G_{s_2-s_1}(y-x)G_{t/\eps^2-s_2}(\frac{w}{\eps}-y)G_{t/\eps^2}(\frac{w}{\eps})^{-1}dxdyds_1ds_2.
\end{aligned}
\]
By the fact that $w\in A$ and $0\nin A$, we have 
\[
\frac{G_{t/\eps^2-s_2}(\frac{w}{\eps}-y)}{G_{t/\eps^2}(\frac{w}{\eps})}=\frac{G_{t-\eps^2s_2}(w-\eps y)}{G_t(w)} \les1,
\]
uniformly in $s_2\leq t/\eps^2, w\in A, \, y\in \mathrm{supp}(R)$ and $\eps\ll1$. This shows that we can remove the conditional expectation and derive
\[
\sup_{w\in A} \E_{B}\left[ \big(\int_{t/\eps^\alpha}^{t/\eps^2}R(B_s)ds\big)^2 \,\bigg|\,   B_{t/\eps^2}=\frac{w}{\eps}\right] \les \E_B\left[ \big(\int_{t/\eps^\alpha}^{t/\eps^2}R(B_s)ds\big)^2\right] \les \eps^{\alpha(\frac{d}{2}-1)},
\]
which completes the proof.
\end{proof}

\section{Negative moments of $Z(t,x)$}
\label{s.nemm}

We now prove Proposition~\ref{p.negative}. 
The goal is to show there exists $\beta_0>0$ such that if $\beta<\beta_0$ and $n\in\Z_+$, we have
\begin{equation}\label{e.negmm}
\sup_{t>0} \E[ Z(t,x)^{-n}] \leq C_{\beta,n},
\end{equation}
with some constant $C_{\beta,n}>0$. We adapt to our setting 
the proof of \cite[
Corollary 4.8]{hu2018asymptotics}, which deals with the case 
when the noise is also singular in space. The same proof applies to our situation, 
and we only present the details for the convenience of the readers.

Since $Z(t,x)$ has the same distribution as $u(t,x)$, it suffices to 
estimate the small ball probability~$\Pb[u(t,x)\leq r]$ for $r\ll 1$. 
We define an approximation of the spacetime white noise
\[
W_\eps(t,x)=e^{-\eps(t^2+|x|^2)}\int_{\R^{d+1}} \phi_\eps(t-s,x-y)dW(s,y),
\]
where $\phi_\eps(t,x)=\eps^{-d-2}\phi(t/\eps^2,x/\eps)$ with $\phi\in\C_c^\infty(\R^{d+1})$ such that $\phi\ge 0$ is even and $\|\phi\|_{L^1}=1$. It is clear that for fixed $\eps>0$,  
$W_\eps\in L^2(\R^{d+1})\cap \C^\infty(\R^{d+1})$ almost surely. We will use $\|\cdot\|_2$ to denote the $L^2(\R^{d+1})$ norm. Define 
\[
V_\eps(t,x)=\int_{\R^d} \varphi(x-y)W_\eps(t,y)dy, \    \   \ccR_\eps(t,s,x,y)=\E[V_\eps(t,x)V_\eps(s,y)],
\]
 and 
\[
\cU_\eps(t,x)=\E_B\left[ e^{\V_t^\eps(B) }\right],
\]
with 
\[
\V_t^\eps(B)=\beta\int_0^t V_\eps(t-s,x+B_s)ds-\frac12\beta^2\Q_\eps(t,x,x,B,B),
\]
where
\[
\Q_\eps(t,x,y,B^1,B^2)=\int_0^t\int_0^t  \ccR_\eps(t-s,t-\ell,x+B_s^1,y+B_{\ell}^2)dsd\ell.
\]
By \cite[Proposition 4.2]{hu2018asymptotics}, $\cU_\eps(t,x)\to u(t,x)$ in probability so we only need to estimate $\Pb[\cU_\eps(t,x)\leq r]$ for $r\ll1$. 

With any given $W_\eps$, define the expectation
\[
\E^{W_\eps}_B[ F(B^1,B^2)]=\frac{\E_B[ F(B^1,B^2)e^{\V_t^\eps(B^1)+\V_t^\eps(B^2)}]}{\E_B[ e^{\V_t^\eps(B^1)+\V_t^\eps(B^2)}]}.
\]
To emphasize the dependence of $\cU_\eps$ on $W_\eps$, we write $\cU_\eps(t,x)=\cU_\eps(t,x,W_\eps)$. For any $\lambda>0$, define the set 
\[
A_\lambda(t,x)=\left\{W_\eps: \cU_\eps(t,x,W_\eps)>\frac12,  \int_0^t\E_B^{W_\eps}[ R(B^1_s-B^2_s)]ds\leq \lambda\right\}.
\]

\begin{lemma}
For any $\tilde{W}_\eps\in A_\lambda(t,x)$, we have 
\[
\cU_\eps(t,x,W_\eps)\geq \frac12e^{-\sqrt{\lambda} \|W_\eps-\tilde{W}_\eps\|_2},
\]
with $\|\cdot\|_2$ denoting the $L^2(\R^{d+1})$ norm.
\end{lemma}

\begin{proof}
We write 
\[
\begin{aligned}
\cU_\eps(t,x,W_\eps)=\E_B[ e^{\V_t^\eps(B)}]=&\E_B[e^{\tilde{\V}_t^\eps(B)}] \frac{ \E_B[e^{\V_t^\eps(B)-\tilde{\V}_t^\eps(B)} e^{\tilde{\V}_t^\eps(B)}]}{\E_B[ e^{\tilde{\V}_t^\eps(B)}]}\\
=&\cU_\eps(t,x,\tilde{W}_\eps) \E_B^{\tilde{W}_\eps}[ e^{\V_t^\eps(B)-\tilde{\V}_t^\eps(B)}],
\end{aligned}
\]
where $\tilde{\V}_t^\eps$ is obtained by replacing $W_\eps\mapsto \tilde{W}_\eps$ in the expression of $\V_t^\eps$. Since $\tilde{W}_\eps\in A_\lambda$, by Jensen's inequality we have 
\[
\cU_\eps(t,x,W_\eps)\geq\frac12  \exp\left(\E_B^{\tilde{W}_\eps}[\V_t^\eps(B)-\tilde{\V}_t^\eps(B)]\right).
\]
It remains to show  that
\begin{equation}\label{e.se251}
|\E_B^{\tilde{W}_\eps}[\V_t^\eps(B)-\tilde{\V}_t^\eps(B)]| \leq \sqrt{\lambda}\|W_\eps-\tilde{W}_\eps\|_2.
\end{equation}
We write 
\[
\begin{aligned}
\V_t^\eps(B)-\tilde{\V}_t^\eps(B)=&\beta\int_0^t [V_\eps(t-s,x+B_s)-\tilde{V}_\eps(t-s,x+B_s)]ds\\
=&\beta \int_0^t\int_{\R^d} \varphi(x+B_s-y)[W_\eps(t-s,y)-\tilde{W}_\eps(t-s,y)]dyds,
\end{aligned}
\]
and apply the Cauchy-Schwarz inequality to get
\[
\begin{aligned}
|\E_B^{\tilde{W}_\eps}[\V_t^\eps(B)-\tilde{\V}_t^\eps(B)]|
&\leq \beta\|W_\eps-\tilde{W}_\eps\|_2\sqrt{\int_0^t \int_{\R^d} |\E_B^{\tilde{W}_\eps}[\varphi(x+B_s-y)] |^2 dyds} \\
&\leq \beta\|W_\eps-\tilde{W}_\eps\|_2\sqrt{\int_0^t \E_B^{\tilde{W}_\eps}[R(B^1_s-B^2_s)]ds}\leq \sqrt{\lambda} \|W_\eps-\tilde{W}_\eps\|_2,
\end{aligned}
\]
which completes the proof.
\end{proof}

\begin{lemma}\label{l.lowerbd}
There exists universal constants $\lambda,c>0$ such that $\Pb[A_\lambda(t,x)]\geq c$.
\end{lemma}

\begin{proof}
We have 
\[
\Pb[A_\lambda(t,x)] \geq \Pb[\cU_\eps(t,x,W_\eps)>
\frac{1}{2}]-\Pb[ B_\lambda(t,x)],
\]
with 
\[
B_\lambda(t,x)=\left\{W_\eps:  \cU_\eps(t,x,W_\eps)>\frac12,  \int_0^t\E_B^{W_\eps}[ R(B^1_s-B^2_s)]ds> \lambda\right\}.
\]
Using the fact that $\E[\cU_\eps(t,x,W_\eps)]=1$ and the Paley-Zygmund's inequality, we have 
\[
\Pb\Big[\cU_\eps(t,x,W_\eps)>\frac{1}{2}\Big]
\geq \frac{1}{4\E[\cU_\eps(t,x,W_\eps)^2]}=\frac{1}{4\E_B[ e^{\beta^2 \Q_\eps(t,x,x,B^1,B^2)}]}.
\]
For $B_\lambda(t,x)$, we have, as $\cU_\eps(t,x,W_\eps)>1/2$, 
\begin{equation}\label{nov2702}
\begin{aligned}
\Pb[B_\lambda(t,x)]\leq &\Pb\left[ \int_0^t \E_B[ R(B_s^1-B_s^2) 
e^{\V_t^\eps(B^1)+\V_t^\eps(B^2)}] ds >\frac{\lambda}{4}\right] \\
\leq &\frac{4}{\lambda}\E_B\left[ e^{\beta^2 
\Q_\eps(t,x,x,B^1,B^2)}\int_0^t R(B^1_s-B^2_s)ds\right]
\leq \frac{4C}{\lambda}\E_B\left[ e^{2\beta^2 \Q_\eps(t,x,x,B^1,B^2)}\right]^{1/2},
\end{aligned}
\end{equation}
with some constant $C>0$. By Lemma~\ref{l.intersection} below and choosing $\lambda$ large, there exists some constants $c,\lambda>0$ independent of $\eps,t,x$ such that $\Pb[A_\lambda(t,x)] \geq c$.
\end{proof}

\begin{lemma}\label{l.intersection}
There exists $\beta_0>0$ such that if $\beta<\beta_0$, we have 
\[
1\leq\E_B\left[ e^{\beta \Q_\eps(t,x,x,B^1,B^2)}\right]\leq C_\beta.
\]
\end{lemma}

\begin{proof}
Recall that 
\[
\Q_\eps(t,x,x,B^1,B^2)=\int_0^t\int_0^t  \ccR_\eps(t-s,t-\ell,x+B_s^1,x+B_{\ell}^2)dsd\ell.
\]
We write $\ccR_\eps$ explicitly:
\[
\begin{aligned}
\ccR_\eps(t_1,t_2,x_1,x_2)=&\int_{\R^{2d}} \varphi(x_1-y_1)\varphi(x_2-y_2)\E[W_\eps(t_1,y_1)W_\eps(t_2,y_2)]dy_1dy_2\\
=&\int_{\R^{2d}} \varphi(x_1-y_1)\varphi(x_2-y_2) e^{-\eps(t_1^2+t_2^2+|y_1|^2+|y_2|^2)} \phi_\eps\star\phi_\eps(t_1-t_2,y_1-y_2) dy_1dy_2\\
\leq &\int_{\R^{2d}} \varphi(x_1-y_1)\varphi(x_2-y_2) \phi_\eps\star\phi_\eps(t_1-t_2,y_1-y_2) dy_1dy_2,
\end{aligned}
\]
with $\star$ denoting the convolution. By the fact that $\varphi,\phi$ have compact supports, we have 
\[
\ccR_\eps(t_1,t_2,x_1,x_2) \les \eps^{-2}  \1_{|x_1-x_2|\leq C,|t_1-t_2|\leq C\eps^2}.
\]
for some $C>0$. Thus, $\Q_\eps$ is essentially measuring the mutual ``intersection'' time of $B^1,B^2$. By \cite[Corollary 4.4]{GRZ17} and the fact that $d\geq 3$, the proof is complete.
\end{proof}

Now we can write 
\begin{equation}\label{e.se261}
\begin{aligned}
\Pb[\cU_\eps(t,x,W_\eps)\leq r] \leq &
\Pb\Big[\frac{1}{2} e^{-\sqrt{\lambda} \mathrm{dist}(W_\eps,A_\lambda(t,x))}\leq r\Big]
\leq \Pb\left[\mathrm{dist}(W_\eps,A_\lambda(t,x)) \geq \frac{\log (2r)^{-1}}{\sqrt{\lambda}} \right],
\end{aligned}
\end{equation}
where $\mathrm{dist}(W_\eps,A_\lambda(t,x))=\inf\{ \|W_\eps-\tilde{W}_\eps\|_2: \tilde{W}_\eps\in A_\lambda(t,x)\}$. Now we can apply \cite[Lemma 4.5]{hu2018asymptotics} to obtain 
\begin{equation}\label{e.se262}
\Pb\left[\mathrm{dist}(W_\eps,A_\lambda(t,x)) \geq \tau+2 \sqrt{\log\frac{2}{c}} \right] \leq  2e^{-\tau^2/4}
\end{equation}
for all $\tau>0$, where $\lambda,c>0$ are chosen as in Lemma~\ref{l.lowerbd}. Combining \eqref{e.se261} and \eqref{e.se262}, we have 
\[
\Pb[\cU_\eps(t,x,W_\eps)\leq r] \leq  2\exp\left( -\frac14\left(\frac{\log (2r)}{\sqrt{\lambda}}+2\sqrt{\log \frac{2}{c}}\right)^2\right),
\]
which implies $\E[\cU_\eps(t,x,W_\eps)^{-n}]\les1$ and completes the proof of \eqref{e.negmm}.

\end{document}